\documentclass[12pt]{amsart}

\RequirePackage{mathtools}
\RequirePackage{amsopn}
\RequirePackage{amsfonts}
\RequirePackage{paralist}
\RequirePackage{amssymb}
\RequirePackage{amsthm}
\RequirePackage{mathrsfs}
\usepackage[alphabetic]{amsrefs}
\renewcommand\MR[1]{\relax} 
\usepackage{tikz}
\usetikzlibrary{cd,decorations.pathmorphing,decorations.markings}
\mathtoolsset{showonlyrefs} 
\newtheorem{thm}{Theorem}[section]
\numberwithin{equation}{section}

\newtheorem{cor}[thm]{Corollary}
\newtheorem{lemma}[thm]{Lemma}
\newtheorem{prop}[thm]{Proposition}

\theoremstyle{definition}
\newtheorem{definition}[thm]{Definition}

\theoremstyle{remark}
\newtheorem{remark}[thm]{Remark}
\newtheorem{example}[thm]{Example}

\newtheorem{mycomment}[thm]{Comment}
{\end{mycomment}\endgroup}
\def\mathcs{C^{*}}
\newcommand{\cs}{\ensuremath{\mathcs}}

\DeclareMathSymbol{\rtimes}{\mathbin}{AMSb}{"6F}

\def\ibind#1{\mathop{#1\mathord{\mathop{\text{--}}}}\!\Ind\nolimits}
\newcommand\xind{\ibind\X}

\newcommand\C{\mathbf{C}}
\newcommand\T{\mathbf{T}}
\newcommand\Z{\mathbf{Z}}

\newcommand\N{\mathbf{N}}

\newcommand\set[1]{\{\,#1\,\}}
\newcommand\sset[1]{\{#1\}}
\let\tensor=\otimes
\def\restr#1{|_{{#1}}}
\makeatletter
\def\labelenumi{\textnormal{(\@alph\c@enumi)}}
\def\theenumi{\@alph \c@enumi}
\def\labelenumii{\textnormal{(\@roman\c@enumii)}}
\def\theenumii{\@roman \c@enumii}
\newcount\charno
\def\alphapart#1{\charno=96
\advance\charno by#1\char\charno}

\makeatother

\def\<{\langle}
\def\>{\rangle}
\let\ipscriptstyle=\scriptscriptstyle
\def\lipsqueeze{{\mskip -3.0mu}}
\def\ripsqueeze{{\mskip -3.0mu}}
\def\ipcomma{\nobreak\mathrel{,}\nobreak}
\newbox\ipstrutbox
\setbox\ipstrutbox=\hbox{\vrule height8.5pt
width 0pt}
\def\ipstrut{\copy\ipstrutbox}
\def\lip#1<#2,#3>{\mathopen{\relax_{\ipstrut\ipscriptstyle{
#1}}\lipsqueeze
\langle} #2\ipcomma #3 \rangle}
\def\blip#1<#2,#3>{\mathopen{\relax_{\ipstrut
\ipscriptstyle{ #1}}\lipsqueeze\bigl\langle} #2\ipcomma #3 \bigr\rangle}
\def\rip#1<#2,#3>{\langle #2\ipcomma #3
\rangle_{\ripsqueeze\ipstrut\ipscriptstyle{#1}}}
\def\brip#1<#2,#3>{\bigl\langle #2\ipcomma #3
\bigr\rangle_{\ripsqueeze\ipstrut\ipscriptstyle{#1}}}
\def\angsqueeze{\mskip -6mu}
\def\smangsqueeze{\mskip -3.7mu}
\def\trip#1<#2,#3>{\langle\smangsqueeze\langle #2\ipcomma #3
\rangle\smangsqueeze\rangle_{\ripsqueeze\ipstrut\ipscriptstyle{#1}}}
\def\btrip#1<#2,#3>{\bigl\langle\angsqueeze\bigl\langle #2\ipcomma
#3
\bigr\rangle
\angsqueeze\bigr\rangle_{\ripsqueeze\ipstrut\ipscriptstyle{#1}}}
\def\tlip#1<#2,#3>{\mathopen{\relax_{\ipstrut\ipscriptstyle{
#1}}\lipsqueeze \langle\smangsqueeze\langle} #2\ipcomma #3
\rangle\smangsqueeze\rangle}
\def\btlip#1<#2,#3>{\mathopen{\relax_{\ipstrut\ipscriptstyle{
#1}}\lipsqueeze
\bigl\langle\angsqueeze\bigl\langle} #2\ipcomma #3
\bigr\rangle\angsqueeze\bigr\rangle}

\def\ip(#1|#2){(#1\mid #2)}
\def\bip(#1|#2){\bigl(#1 \mid #2\bigr)}
\def\Bip(#1|#2){\Bigl( #1 \bigm| #2 \Bigr)}
\expandafter\ifx\csname BibSpec\endcsname\relax\else
\BibSpec{collection.article}{%
    +{}  {\PrintAuthors}                {author}
    +{,} { \textit}                     {title}
    +{.} { }                            {part}
    +{:} { \textit}                     {subtitle}
    +{,} { \PrintContributions}         {contribution}
    +{,} { \PrintConference}            {conference}
    +{}  {\PrintBook}                   {book}
    +{,} { }                            {booktitle}
    +{,} { }                            {series}
    +{,} { \voltext}                    {volume}
    +{,} { }                            {publisher}
    +{,} { }                            {organization}
    +{,} { }                            {address}
    +{,} { \PrintDateB}                 {date}
    +{,} { pp.~}                        {pages}
    +{,} { }                            {status}
    +{,} { \PrintDOI}                   {doi}
    +{,} { available at \eprint}        {eprint}
    +{}  { \parenthesize}               {language}
    +{}  { \PrintTranslation}           {translation}
    +{;} { \PrintReprint}               {reprint}
    +{.} { }                            {note}
    +{.} {}                             {transition}
}
\BibSpec{article}{%
    +{}  {\PrintAuthors}                {author}
    +{,} { \textit}                     {title}
    +{.} { }                            {part}
    +{:} { \textit}                     {subtitle}
    +{,} { \PrintContributions}         {contribution}
    +{.} { \PrintPartials}              {partial}
    +{,} { }                            {journal}
    +{}  { \textbf}                     {volume}
    +{}  { \PrintDatePV}                {date}
    +{,} { \eprintpages}                {pages}
    +{,} { }                            {status}
    +{,} { \PrintDOI}                   {doi}
    +{,} { available at \eprint}        {eprint}
    +{}  { \parenthesize}               {language}
    +{}  { \PrintTranslation}           {translation}
    +{;} { \PrintReprint}               {reprint}
    +{.} { }                            {note}
    +{.} {}                             {transition}
}
\BibSpec{book}{%
    +{}  {\PrintPrimary}                {transition}
    +{,} { \textit}                     {title}
    +{.} { }                            {part}
    +{:} { \textit}                     {subtitle}
    +{,} { \PrintEdition}               {edition}
    +{}  { \PrintEditorsB}              {editor}
    +{,} { \PrintTranslatorsC}          {translator}
    +{,} { \PrintContributions}         {contribution}
    +{,} { }                            {series}
    +{,} { \voltext}                    {volume}
    +{,} { }                            {publisher}
    +{,} { }                            {organization}
    +{,} { }                            {address}
    +{,} { pp.~}                        {pages}
    +{,} { \PrintDateB}                 {date}
    +{,} { }                            {status}
    +{}  { \parenthesize}               {language}
    +{}  { \PrintTranslation}           {translation}
    +{;} { \PrintReprint}               {reprint}
    +{.} { }                            {note}
    +{.} {}                             {transition}
}
\fi
\evensidemargin=\oddsidemargin
\newcommand\go{G^{(0)}}
\newcommand\so{\Sigma_{0}}
\newcommand\X{\mathsf{X}}
\newcommand\xo{\X_{0}}
\newcommand\cc{C_{c}}
\newcommand\Rip{\rip\star}
\newcommand\Ind{\operatorname{Ind}}
\newcommand\indgh{\Ind_{H}^{G}}
\renewcommand\H{\mathcal{H}}
\newcommand\atensor{\odot}

\renewcommand\L{\mathcal{L}}
\newcommand\Sigmah{\widehat\Sigma}
\newcommand\Sigmac{\Sigmah_{S}}
\newcommand\indsg{\Ind_{\Sigma}^{G}}

\newcommand\stab{\operatorname{Stab}}
\newcommand\stabg{\stab(G)}
\newcommand\rs{\rho_{S}}
\newcommand\ts{\tau_{S}}
\newcommand\Prim{\operatorname{Prim}}
\newcommand\primcsg{\Prim\cs(G)}

\newcommand\psind{\kappa}
\newcommand\psindq{\tilde\psind}
\newcommand\irrind{\Ind}
\newcommand\irrindq{\underline{\irrind}}

\newcommand\stabgg{G\backslash\stabg}
\newcommand\stabggt{(\stabgg)^{\sim}}

\newcommand\gugo{G\backslash\go}

\renewcommand\O{\mathcal{O}}

\newcommand\I{\mathcal{I}}

\newcommand\sg{\mathbb{H}}
\newcommand\sgr{\sg(r(\gamma))}
\newcommand\sgu{\sg(u)}

\def\ipu(#1|#2){\bip({#1}|{#2})_{u}}

\newcommand\uL{\underline{L}}

\newcommand\grg{\mathcal{G}}
\newcommand\grh{\mathcal{H}}
\newcommand\lt{\operatorname{lt}}

\newcommand\hath{\hat\grh}

\newcommand\rg{R_{G}}

\newcommand\Nz{\N_{0}} 
\newcommand\dom{\operatorname{dom}}
\newcommand\gto{G_{T}^{(0)}}

\newcommand\stabgt{\stab(G_{T})}

\newcommand\A{\mathcal{A}}
\newcommand\G{\mathcal{G}}

\allowdisplaybreaks
\begin{document}
\begin{abstract}
  We study the topology of the primitive ideal space of groupoid
  \cs-algebras for groupoids with abelian isotropy.  Our results
  include the known results for action groupoids with abelian
  stabilizers.  Furthermore, we obtain complete results when the
  isotropy map is continuous except for jump discontinuities, and also
  when $G$ is a unit space fixing extension of a proper groupoid by
  an abelian group bundle. We hope that our methods will be a
  springboard to further results of this type.
\end{abstract}

\title[Primitive Ideal Space]{The Primitive Ideal Space of Groupoid
  \cs-Algebras for 
  Groupoids with Abelian Isotropy}

\author[van Wyk]{Daniel W. van Wyk}
\address{Department of Mathematics\\ Dartmouth College \\ Hanover, NH
  03755-3551 USA}
\email{dwvanwyk79@gmail.com}

\author[Williams]{Dana P. Williams}
\address{Department of Mathematics\\ Dartmouth College \\ Hanover, NH
  03755-3551 USA}
\email{dana.williams@Dartmouth.edu}

\date{24 August 2021}

\maketitle


\section*{Introduction}
\label{sec:intro}

If $(\grg,X)$ is a second countable locally compact transformation
group, then we can form the action groupoid $G=G(\grg,X)$.  The term
transformation groupoid is also used as in this case the groupoid
\cs-algebra, $\cs(G)$, is naturally isomorphic to the transformation
group \cs-algebra $C_{0}(X)\rtimes_{\lt}\grg$.  If all the stability
groups $S_{x}=\set{g\in \grg:g\cdot x=x}$ are contained in a fixed
abelian subgroup of $\grg$, and if $(\grg,X)$ is EH-regular, then it
has been known for a long time how describe both the primitive ideal
space, $\primcsg$, of $\cs(G)$ as well as its Jacobson topology
\cite{wil:tams81}*{Theorem~5.8}.  This immediately suggests the question
as to what one can say about $\primcsg$ if we merely assume that each
stability group is abelian.  Since transformation groups provide
fundamental examples of groupoids, this naturally leads to considering
locally compact groupoids with abelian isotropy groups.  At present, a
general result of this type is well out of reach even for action
groupoids.  This is amply illustrated by
\cite{simwil:art15}*{Theorem~3.2} where considerable technology is
required just to describe $\primcsg$ as a set when $G$ is a second
countable locally compact Deaconu--Renault groupoid.

The purpose of this paper is to suggest an attack on the problem of
describing $\primcsg$ as a topological space when $G$ is a second
countable locally compact Hausdorff groupoid and all the isotropy
groups $G(u)=\set{\gamma\in G:s(\gamma)=u=r(\gamma)}$ are abelian.  We
are able to obtain results that subsume the known results for action
groupoids with abelian stability groups, and to extend these results
to the groupoid setting.  We hope that this work will motivate even
more successful attacks on the problem in the future.  
In an expository vein, we also give an alternate description of $\primcsg$
in cases where the answer is already known.  For our approach to work,
we need to know that our groupoid $G$ is EH-regular in the sense that
every primitive ideal in $\primcsg$ is induced from a character $\pi$
of a isotropy group $G(u)$ with $u\in\go$.  Fortunately, this is known
to be the case if if every
$G$-orbit $[u]:=G\cdot u$ is locally closed in $\go$ 
\cite{wil:toolkit}*{Theorem~5.35}, or more generally if
$G$ is amenable \cite{ionwil:iumj09}. Then we can parameterize
$\primcsg$ as a quotient of $\stabg=\set{(u,\pi):\text{$u\in\go$ and
    $\pi\in G(u)^{\wedge}$}}$ where $G(u)^{\wedge}$ is the Pontryagin
dual of the abelian group $G(u)$.  Having described this quotient, the
crucial step is to equip it with a topology such that it is
homeomorphic to $\primcsg$.  We also want the description of this
topology to be sufficiently transparent so that it can be used in
applications.  It should be said that our methods are unlikely to uncover
explicit information unless one can describe the isotropy groups in a
coherent way.

In the case of an action groupoid $G=G(\grg,X)$, with $\grg$ acting
properly on $X$, Echterhoff and Emerson solve this problem elegantly
in \cite{echeme:em11}*{Theorem~4.6} without assuming that the isotropy
is abelian provided $(\grg,X)$ satisfies what they call the Palais
slice property.  (This additional assumption is automatically
satisfied when $\grg$ is a Lie group or a discrete group.)  A related,
but slightly different, solution for all proper group actions is
presented in Katharina Neumann's Ph.D.\ thesis \cite{neu:phd11}.
Neumann replaces the Palais slice property with a theorem of Abels
\cite{abe:mz78}*{Theorem~3.3} for proper actions.  Her
description of the topology requires knowledge of the Jacobson
topology on the spectrum of the Fell subgroup algebra
$\mathcal S(\grg)$.  
In the extreme case that the isotropy map
$u\mapsto G(u)$ is continuous, Geoff Goehle gives a
description of $\primcsg$ as a topological space for any groupoid with
abelian isotropy in \cite{goe:rmjm12}*{Theorem~3.5}---in fact,
Goehle's results apply to the much more general setting of groupoid
crossed products \cite{goe:rmjm12}*{Theorem~2.23}.

In this paper, we equip $\stabg$ with a second countable topology so
that the natural $G$-action is continuous.  As does Neumann, this
description is in terms of an auxillary space, but our space is the
Gelfand dual of an abelian \cs-algebra, and hence is much more easily
and concretely described. In the case that $\gugo$
is $T_{0}$---that is, each orbit $[u]$ is locally closed in
$\go$---induction induces a continuous injection $\irrindq$ of
$\stabgg$ onto $\cs(G)^{\wedge}$ (Corollary~\ref{cor-main-cts}).  In the
general case, we need to consider the quotient $\stabggt$ where we
identify orbits with the same closure.  Then we obtain a continuous
map $\psindq$ of $\stabggt$ into $\primcsg$ which is surjective if $G$
is EH-regular (Proposition~\ref{prop-main-cts}).  Naturally, we expect
that both $\irrindq$ and $\psindq$ are homeomorphisms in most
circumstances.  Unfortunately, at present, we can only prove this
under somewhat restrictive hypotheses.  In fact, we aren't even sure
whether $\psindq$ is always injective.  Our hope is that we, or
others, will be able to push the envelope in future work.

Nevertheless, we are able to make some interesting new contributions.
In particular, we consider groupoids $G$ where the isotropy map is
continuous except for jump discontinuities at discrete points.  Our
first main result, Theorem~\ref{thm-jmp-dis}, is that $\psindq$ is a
homeomorphism in this case.  Thus we obtain Goehle's continuously
varying isotropy result as a special case, and also give an elegant
description of Goehle's topology on $\stabggt$.  We show by example that
our result applies to the groupoids associated to many directed
graphs.

Our second main result, Theorem~\ref{thm-baby}, shows that $\irrindq$
is a homeomorphism whenever $G$ is what we call ``proper modulo its
isotropy''.  This class of groupoids not only includes proper
groupoids but also extensions of proper groupoids by abelian group bundles
as studied in, for example,
\cites{iksw:jot19,ikrsw:jfa20,ikrsw:xx21}.  In
particular, the isotropy groups need no longer be 
compact in this case.  This result
recovers the Echter\-hoff--Em\-er\-son
and Neumann results in the case $G=G(\grg,X)$ with $\grg$ acting
properly with abelian stabilizers.  

The paper is organized as follows.  In Sections
\ref{sec:preliminiaries}~and \ref{sec:subgroup-module}, we set out our
assumptions and review some standard constructions including induced
representations of groupoids.  In Section~\ref{sec:stabg}, we
introduce our topology on $\stabg$.  In addition to giving a criterion
for convergence in this topology, we provide a basis for our topology
and investigate
some of $\stabg$'s elementary properties.

In Section~\ref{sec:primcsg} we introduce the maps $\irrindq$ and
$\psindq$ described above, and establish continuity.  Even though
these are not (yet) homeomorphism results, we allow ourselves an aside
in Section~\ref{sec:rena-simpl-result} to see that our methods already
recover a well known simplicity result of Renault's for the special case of
amenable groupoids with abelian isotropy.

In Section~\ref{sec:subgroup-sections}, we introduce our basic tool,
what we call a subgroup section, which allows us to deal with certain
discontinuities of the isotropy map.  We use this in
Section~\ref{sec:continuous-isotropy} to prove
Theorem~\ref{thm-jmp-dis} and thereby recover Goehle's result as a
special case.

In Section~\ref{sec:abelian-group-action}, we show that our
description includes the known results for action groupoids.  At the
same time, we obtain a separate description of the Jacobson topology
in these cases which, depending on the application at hand, might be a
more useful description.

In Section~\ref{sec:gener-prop-group} we prove our
Theorem~\ref{thm-baby}
which
generalizes the Echter\-hoff--Em\-er\-son/Neumann result to proper
groupoids in the 
case of abelian isotropy and also applies to groupoids that are proper
modulo their isotropy.   We  provide some interesting examples in
Section~\ref{sec:examples} of groupoids (other than action groupoids)
with abelian isotropy that is continuous except for jump
discontinuities.   

\subsubsection*{Assumptions}
\label{sec:assumptions}

Throughout, $G$ will be a second countable locally compact Hausdorff
groupoid with abelian isotropy and a Haar system
$\lambda=\set{\lambda^{u}}_{u\in\go}$.  Similarly, locally compact
spaces are assumed to be second countable and Hausdorff.  We assume
that our \cs-algebras are separable with the exception of multiplier
algebras which obviously aren't.  Homomorphisms between \cs-algebras
are assumed to be $*$-preserving and representations are assumed to be
nondegenerate.  We rely on \cite{wil:toolkit} as a convenient
references for basic groupoid technology and notation.  We also rely
heavily on the Rieffel machine for inducing representations as laid
out in \cite{rw:morita}*{\S2.4}.

\section{Preliminiaries}
\label{sec:preliminiaries}

As in \cite{wil:toolkit}*{\S3.4}, we let $\so$ be the locally compact
Hausdorff space of closed subgroups of $G$ with the Fell topology.  If
$H\in \so$, then there is a $u=p_{0}(H)\in \go$ such that $H$ is a
subgroup of the isotropy group
$G(u):=\set{\gamma\in G:s(\gamma)=u=r(\gamma)}$.  Then
$p_{0}:\so\to \go$ is continuous and proper in the sense that
$p_{0}^{-1}(K)$ is compact for all compact $K\subset \go$.

Let $\pi$ be a representation of $H\in\so$.  As is usual for
groups, we do 
not distinguish between the unitary representation $\pi:G\to
U(\H_{\pi})$ and its integrated form $\pi:\cc(H)\to B(\H_{\pi})$ given
by
\begin{equation}
  \label{eq:7}
  \pi(b)h=\int_{H}b(t)\pi(t)h\,d\beta^{H}(t)\quad\text{$b\in\cc(H)$, and
$h\in\H_{\pi}$.}
\end{equation}

There is a well known method for inducing a representation $\pi$ of
$H\in\so$ to a representation $\indgh\pi$ of $\cs(G)$ based on the
Rieffel machine.  The details are spelled out in
\cite{ionwil:pams08}*{\S2}, and we recall some of these to motivate
the constructions in Section~\ref{sec:subgroup-module}.  For
convenience, let $u=p_{0}(H)$ and let $G_{u}:=s^{-1}(u)$.  We equip
$\cc(G_{u})$ with a $\cc(H)$-valued pre-inner product given by
\begin{equation}
  \label{eq:8}
  \Rip<\psi,\phi>(t)=\int_{G}\overline{\psi(\gamma)}\phi(\gamma t)
  \,d\lambda_{u}(\gamma).
\end{equation}
Then we can complete $\cc(G_{u})$ to a right Hilbert $C_{c}(H)$-module
$\X$.  We get an action of $\cs(G)$ on $\X$ by adjointable operators
where $f\in \cc(G)$ acts on $\phi\in\cc(G_{u})$ by 
\begin{equation}
  \label{eq:4}
  f\cdot
  \phi(\gamma)
  =\int_{G}f(\eta)\phi(\eta^{-1}\gamma)\,d\lambda^{r(\gamma)}(\eta). 
\end{equation}
Then the induced representation
$\indgh\pi$ of $\cs(G)$ acts on $\H_{\Ind\pi}$ which is the completion of
$\cc(G_{u})\atensor \H_{\pi}$ with respect to the pre-inner product
\begin{align}
  \label{eq:3}
  \ip(\phi\tensor h|\psi\tensor k)
  &= \bip(\pi\bigl(\Rip<\psi,\phi>\bigr)h|k) 
    \,d\beta^{H}(t) \\
  &=\int_{H}\int_{G}\overline{\psi(\gamma)}\phi(\gamma t) \bip({\pi(t)h}|k)
   \,d\lambda_{u}(\gamma)\,d\beta^{H}(t).
\end{align}
If $f\in \cc(G)$, then $\indgh\pi(f)$ acts on
$\cc(G_{u})\atensor \H_{\pi}$ by $\indgh\pi(f)(\phi\tensor h)=f\cdot
\phi\tensor h$.

In the sequel, we will be interested primarily in the case that $\pi$
is a character in the dual $\hat H$.  Then $\H_{\pi}=\C$ and we can
identify $\cc(G_{u})\atensor \H_{\pi}$ with $\cc(G_{u})$ equipped with
the inner product
\begin{equation}
  \label{eq:5}
  \ip(\phi|\psi) =\int_{H}\int_{G}\overline{\psi(\gamma)}\phi(\gamma
  t) \pi(t)\,d\lambda_{u}(\gamma)\,d\beta^{H}(t).
\end{equation}

We can form a locally compact abelian group bundle $p:\Sigma\to\so$ where
\begin{equation}
  \label{eq:1}
  \Sigma=\set{(H,t)\in\so\times G:t\in H}.
\end{equation}
Then  $\Sigma$ always has a Haar system
$\beta=\set{\beta^{H}}_{H\in\so}$ by
\cite{wil:toolkit}*{Proposition~3.23}.\footnote{In
  \cite{wil:toolkit}*{\S3.5}, the elements of $\Sigma$ were written as
triples for pedagogical reasons.  Here we have opted for a more
compact notation.}   It will be useful to recall the following.

\begin{lemma}[\cite{wil:toolkit}*{Lemma~3.25}]   
  \label{lem-3.25} Let $G*\so=\set{(\gamma,H):s(\gamma)=p_{0}(H)}$.
  There is a continuous
  function $\omega:G*\so\to (0,\infty)$ such that
  \begin{equation}
    \label{eq:21}
    \int_{H}f(\sigma \gamma\sigma^{-1}) \,d\beta^{H}(\gamma)
    =\omega(\sigma,H)\int_{\sigma\cdot H} f(\gamma)
    \,d\beta^{\sigma\cdot H}(\gamma)
  \end{equation}
  for all $f\in\cc(G)$.   Moreover, for all $\sigma,\tau\in G$ and
  $H\in\so$, we have
  \begin{equation}
    \label{eq:22}
    \omega(\sigma\tau,H)=\omega(\tau,H)\omega(\sigma,\tau\cdot
    H)\quad\text{and} \quad
    \omega(\sigma,H)^{-1}=\omega(\sigma^{-1},\sigma\cdot H).
  \end{equation}
\end{lemma}

Since $\Sigma$ is an abelian group bundle,  $\cs(\Sigma)$ is
a commutative \cs-algebra.  As shown in \cite{mrw:tams96}*{\S3}, its
Gelfand spectrum, $\Sigmah$, is an abelian group bundle over $\so$
such that the fibre over $H$ can be identified with $\hat H$.  Thus we
can assume that $\Sigmah=\set{(H,\pi):\text{$H\in\so$ and $\pi\in \hat
  H$}}$.  Then $(H,\pi)$ represents the complex homomorphism given on
$b\in \cc(\Sigma)$  by
\begin{equation}
  \label{eq:10}
  (H,\pi)(b)=\hat b(H,\pi)=\int_{H}b(H,t)\pi(t) \,d\beta^{H}(t).
\end{equation}

Whenever convenient, we identify $\cs(\Sigma)$ with
$C_{0}(\Sigmah)$ via the Gelfand transform $b\mapsto \hat b$, where
$\hat b$ is defined as in \eqref{eq:10}.

One of the benefits of restricting to abelian isotropy is that the
Gelfand topology on $\Sigmah$ is well understood with a convenient and elegant 
description in terms of sequences.

\begin{prop}[\cite{mrw:tams96}*{Proposition~3.3}]
  \label{prop-mrw3.3} A sequence $\bigl((H_{n},\pi_{n})\bigr)$ in
  $\Sigmah$ con\-ver\-ges to $(H_{0},\pi_{0})$ in the Gelfand topology on
  $\Sigmah$ if and only if
  \begin{enumerate}
  \item $H_{n}\to H_{0}$ in $\so$, and
  \item whenever $a_{n}\in H_{n}$ is such that $a_{n}\to a_{0}$, then we have
    $\pi_{n}(a_{n})\to \pi_{0}(a_{0})$.
  \end{enumerate}
\end{prop}

\section{The Subgroup Module}
\label{sec:subgroup-module}

Let $\xo =\cc(G*\so)$.  It is straightforward to check that
\begin{equation}
  \label{eq:2}
  \Rip<F,G>(H,t)=\int_{G}\overline{F(\gamma,H)}G(\gamma
  t,H)\,d\lambda_{p_{0}(H)}(\gamma) 
\end{equation}
is a $\cc(\Sigma)$-valued pre-inner product on $\X_{0}$ so that the
completion is a right Hilbert $\cs(\Sigma)$-module $\X$ (see
\cite{rw:morita}*{Lemma~2.16}).  Furthermore, we get a left action of
$\cc(G)$ on $\xo$ by
\begin{equation}
  \label{eq:9}
  f\cdot
  F(\gamma,H)=\int_{G}f(\eta)F(\eta^{-1}\gamma,H)\,d\lambda^{r(\gamma)}(\eta).  
\end{equation}
Just as in \cite{ionwil:pams08}*{\S2}, this action extends to a
homomorphism of $\cs(G)$ into the algebra of adjointable operators
$\L(\X)$. Therefore we can use the Rieffel machine to induce a
representation $L$ of $\cs(\Sigma)$ to a representation $\indsg(L)$
of $\cs(G)$.  In particular, if $H\in\so$ and $\pi\in\hat H$, then we
can let $L=(H,\pi)$ and form the representation $\indsg (H,\pi)$.
Then we have the following equivalence result.  The proof is routine,
so we omit it.

\begin{lemma}\label{lem-ind-equiv}
  If $H\in \so$ and $\pi\in\hat H$, then $\indgh \pi$ and
  $\indsg(H,\pi)$ are equivalent representations of $\cs(G)$.
\end{lemma}


Recall that if $A$ is a \cs-algebra, then we can equip the collection
$\I(A)$ of closed two-sided ideals of $A$ with the topology generated
by the sets
\begin{equation}
  \label{eq:6}
  \O_{J}=\set{I\in\I(A):J\not\subset I}
\end{equation}
where $J\in\I(A)$.  Therefore a basis for this topology consists of
finite intersections of $\set{\O_{J}:J\in \I(A)}$  with the whole
space arising as in the intersection over the empty set.  One of the
few redeeming properties of this topology is that it restricts to the
usual Jacobson topology on $\Prim A\subset \I(A)$, and then the
collection $\set{\O_{J}:J\in \I(A)}$ is actually a basis for the
Jacobson topology on $\Prim A$.

If $\X$ is a right Hilbert $B$-module and $\phi:A\to \L(\X)$ is a
homormorphism, then the Rieffel machine gives us a means of inducting
a representation $\pi$ of $B$ to a representation $\xind_{B}^{A}\pi$
of $A$.  Furthermore, we get a well-defined continuous map from
$\I(B)$ to $\I(A)$ given by
$\ker\pi\mapsto \ker\bigl(\xind_{B}^{A}\pi\bigr)$
\cite{rw:morita}*{Corollary~3.35}.  Since $\indsg$ arises from such a
Rieffel machine, and since $\Sigmah$ and $\Prim\bigl(\cs(\Sigma)\bigr)$
are homeomorphic, we have the following.

\begin{lemma}\label{lem-ind-cts}
  The map $(H,\sigma)\mapsto \ker\bigl(\indsg(H,\sigma)\bigr)$ is a
  continuous map from $\Sigmah$ to $\I(\cs(G))$.
\end{lemma}

The following observation will be of use in the sequel.

\begin{lemma}[\cite{wil:crossed}*{Lemma~8.38}] \label{lem-fix-8.38}
  Suppose that $X$ is a second countable locally compact Hausdorff
  space.  If $F$ is closed in $X$, let $I(F)$ be the ideal in
  $C_{0}(X)$ of functions vanishing on $F$.  If a sequence
  $\bigl(I(F_{n})\bigr)$ converges to $I(F)$ in
  $\I\bigl(C_{0}(X)\bigr)$ and if $x\in F$, then there is a
  subsequence $(I(F_{n_{k}}))$ and $x_{k}\in F_{n_{k}}$ such that
  $x_{k}\to x$ in $X$.\footnote{This is special case of
    \cite{wil:crossed}*{Lemma~8.38}.  Hence the result remains true
    without a separability assumption if nets are used in place of
    sequences.  It should be noted that the converse statement in
    \cite{wil:crossed}*{Lemma~8.38} is false as stated.  It should be
    amended to say that $I(F_{i})\to I(F)$ if every subnet has the
    given property.}
\end{lemma}
\begin{proof}
  Suppose that $I(F_{n})\to I(F)$ and $x\in F$.  Let
  $\sset{U_{k}}_{k=1}^{\infty} $ be a countable neighborhood basis at
  $x$ consisting of open sets such that $U_{k+1} \subset U_{k}$.  Let
  $J_{k}=I(X\setminus U_{k})$.  Then $I(F)\in \O_{J_{k}}$ for all $k$.
  Thus there is a $n_{1}$ such that $n\ge n_{1}$ implies
  $I(F_{n})\in \O_{J_{1}}$.  In particular, there is a
  $x_{1}\in F_{n_{1}}\cap U_{1}$.  Then if $k\ge 2$ and if we have
  choosen $x_{j}\in F_{n_{j}}$ such that
  $x_{j}\in F_{n_{j}} \cap U_{j}$ for all $j<k$, then we can find
  $n_{k}>n_{k-1}$ such that $I(F_{n_{k}})\in \O_{J_{k}}$.  Hence we
  can find $x_{k}\in F_{n_{k}}\cap U_{k}$.  Then $x_{k}\to x$ as
  required.
\end{proof}

\section{$\stabg$}
\label{sec:stabg}

\begin{definition}
  We let
  \begin{equation}
    \label{eq:13}
    \stabg=\set{(u,\pi):\text{$u\in \go$ and $\pi \in G(u)^{\wedge}$}}.
  \end{equation}
\end{definition}

\begin{remark}
  The map $(u,\pi)\mapsto \bigl(G(u),\pi\bigr)$ identifies $\stabg$
  with a subset of $\Sigmah$, but we need a more subtle topology on
  $\stabg$ than the relative topology.
\end{remark}

We let $G'=\set{\gamma\in G:r(\gamma)=s(\gamma)}$ be the isotropy
subgroupoid of $G$.  Then $G'$ is closed in $G$.  Let $U$ be a
pre-compact open subset of $\go$ and $K$ a compact subset of $G'$.  If
$V$ is open in the unit circle $\T$, then we let
\begin{equation}
  \label{eq:14}
  \mathcal O(U,K,V)=\set{(u,\pi)\in\stabg:\text{$u\in U$ and
      $\pi(G(u)\cap K)\subset V$}}
\end{equation}
with the understanding that if $G(u)\cap K=\emptyset$, then
$\pi(G(u)\cap K) \subset V$ holds vacuously.  Then we let $\rs$ be the
collection of all $\mathcal (U,K,V)$ for $U$, $K$, and $V$ as above.
Since $\rs$ covers $\stabg$, $\rs$ generates a topology $\ts$ for
$\stabg$ with a basis consisting of finite intersections of elements
of $\rs$.

\begin{lemma}
  \label{lem-ts-sec-count} The topology $\ts$ is second countable.
\end{lemma}
\begin{proof}
  Let $\set{U_{n}}$ and $\set{V_{m}}$ be countable bases for $\go$ and
  $\T$, respectively.  Let $\set{W_{j}}$ be a countable basis of
  pre-compact open subsets of $G'$ and let $K_{j}=\overline{W_{j}}$.
  Then the collection $\mathcal C(G')_{f}$ of finite unions of the
  $K_{j}$ is also countable.

  Given $(u,\pi)\in \mathcal O(U_{n},K,V_{m})$, it suffices to find
  $K'\in \mathcal C(G')_{f}$ such that $K\subset K'$ and
  $\pi\bigl(G(u)\cap K'\bigr) \subset V_{m}$ as this implies
  \begin{equation}
    \label{eq:15}
    (u,\pi)\in \mathcal O(U_{n},K',V_{m}) \subset \mathcal
    O(U_{n},K,V_{m}). 
  \end{equation}

  For each $a\in K$, we claim there is a $j$ such that $a\in W_{j}$
  and $\pi\bigl(G(u)\cap K_{j}\bigr)\subset V_{m}$.  If
  $a\notin G(u)$, then by regularity there is a $U\in G'$ such that
  $a\in U\subset \overline U \subset G'\setminus G(u)$.  Then we can
  find $W_{j}$ such that $a\in W_{j}\subset U$ and $K_{j}$ will do.
  If $a\in G(u)$, then the continuity of $\pi$ implies there is a
  neighborhood $V$ of $a$ in $G'$ such that
  $\pi\bigl(G(u)\cap V\bigr) \subset V_{m}$.  Again regularity implies
  there is a $W_{j}$ such that $a\in W_{j}\subset K_{j} \subset V$.
  This proves the claim.

  Since $K$ is compact, there are $j_{1},\dots,j_{l}$ such that
  $K\subset \bigcup K_{j_{s}}$ and
  $\pi\bigl(G(u)\cap K_{j_{s}}\bigr) \subset V_{m}$.  Thus we can let
  $K'=\bigcup K_{j_{s}}$.
\end{proof}

We can now give an elegant sequential characterization of $\ts$ in
terms of our sequential characterization of the Gelfand topology on
$\Sigmah$ in Proposition~\ref{prop-mrw3.3}.

\begin{prop}
  \label{prop-conv-ts} A sequence $\bigl((u_{n},\pi_{n})\bigr)$
  converges to $(u_{0},\pi_{0})$ in $\bigl(\stabg,\ts\bigr)$ if and
  only if every subsequence $\bigl((u_{n_{k}},\pi_{n_{k}})\bigr)$ has
  a subsequence $\bigl((u_{n_{k_{j}}},\pi_{n_{k_{j}}})\bigr)$ such
  that
  \begin{enumerate}
  \item $u_{n_{k_{j}}}\to u_{0}$ in $\go$,
  \item there is a $H\in\so$ such that $G(u_{n_{k_{j}}})\to H $ in
    $\so$, and
  \item
    $\bigl(G(u_{n_{k_{j}}}),\pi_{n_{k_{j}}}\bigr)\to (H,\pi_{0}\restr
    H)$ in $\Sigmah$.
  \end{enumerate}
\end{prop}
\begin{remark}
  Note that the $H$ in part~(b) must satisfy $H\subset G(u_{0})$ and
  can depend on the subsequence.
\end{remark}

\begin{proof}
  Suppose that $(u_{n},\pi_{n})\to (u_{0},\pi_{0})$ in $\stabg$.  If
  $U$ is any pre-compact neighborhood of $u_{0}$, then
  $\mathcal O(U,\sset{u_{0}},\T)$ is a neighborhood of
  $(u_{0},\pi_{0})$.  Hence $u_{n}\to u_{0}$ in $\go$.  Hence (a)
  holds for any subsequence.  We can assume that we have already
  replaced $\bigl((u_{n},\pi_{n})\bigr)$ with a subsequence.  Since
  $p_{0}:\so\to \go$ is a proper map, we can pass to another
  subsequence, relabel, and assume that
  $G(u_{n})\to H \subset G(u_{0})$.  Thus part~(b) will hold for any
  further subsequence.

  Suppose now that $a_{n}\in G(u_{n})$ is such that $a_{n}\to a_{0}$
  in $G$.  Since $G(u_{n})\to H$, we have $a_{0}\in H$.  If it is not
  the case that $\pi_{n}(a_{n})\to \pi_{0}(a_{0})$, then we can pass
  to a subsequence, relabel, and assume that there is a neighborhood
  $V$ of $\pi_{0}(a_{0})$ such that $\pi_{n}(a_{n})\notin V$ for all
  $n$.  Let $W$ be a pre-compact neighborhood of $a_{0}$ in $G$ such
  that $\pi_{0}\bigl(G(u_{0})\cap \overline W\bigr) \subset V$.  Then
  $r(W)$ is a neighborhood of $u_{0}$ in $\go$, and there is a
  pre-compact open set $U$ in $\go$ such that
  $u_{0}\in U \subset \overline U \subset r(W)$.  Then
  $r^{-1}(\overline U)\cap \overline W$ is compact as is
  $K=r^{-1}(\overline U)\cap \overline W\cap G'$.  Since
  $G(u_{0}) \cap K\subset G(u_{0})\cap \overline W$, we have
  $(u_{0},\pi_{0}) \in \mathcal O(U,K,V)$.  Since
  $(u_{n},\pi_{n})\to (u_{0},\pi_{0})$, we eventually have
  $a_{n}\in G(u_{n})\cap K$ and
  $(u_{n},\pi_{n})\in \mathcal O(U,K,V)$.  But then
  $\pi_{n}\bigl(G(u_{n})\cap K\bigr)\subset V$ contradicts
  $\pi_{n}(a_{n})\notin V$.  Therefore
  $\pi_{n}(a_{n})\to \pi_{0}(a_{0})$ and
  $\bigl(G(u_{n}),\pi_{n}\bigr)\to (H,\pi_{0}\restr H)$ by
  Proposition~\ref{prop-mrw3.3}.

  Conversely, suppose that $\bigl((u_{n},\pi_{n})\bigr)$ is a sequence
  such that every subsequence has a subsequence satisfying (a), (b),
  and (c) with respect to $(u_{0},\pi_{0})$.  Suppose that
  $(u_{0},\pi_{0}) \in \mathcal O(U,K,V)$.  It will suffice to see
  that $\bigl((u_{n},\pi_{n})\bigr)$ is eventually in
  $\mathcal O(U,K,V)$.  Suppose not.  Then we can pass to a
  subsequence that is never in $O(U,K,V)$.  Moreover, we can pass to
  additional subsequences if necessary, relabel, and assume that
  $u_{n}\to u_{0}$, $G(u_{n})\to H\in\so$, and that property~(c)
  holds.  We can also assume $u_{n}\in U$ for all $n$.

  Since $(u_{n},\pi_{n}) \notin \mathcal O(U,K,V)$ and $u_{n}\in U$,
  we must have $\pi_{n}\bigl(G(u_{n})\cap K\bigr) \not\subset V$.
  Thus there is a $a_{n}\in G(u_{n})\cap K$ such that
  $\pi(a_{n})\notin V$.  Since $K$ is compact, we can pass to another
  subsequence, relabel, and assume that $a_{n}\to a_{0}$.  Then by
  assumption $\pi_{n}(a_{n})\to \pi_{0}(a_{0})\in V$.  This leads to a
  contradiction and completes the proof.
\end{proof}

As a corollary of the proof, we have the following observation.

\begin{cor}
  \label{cor-good-top1} If $(u_{n},\pi_{n})\to (u_{0},\pi_{0})$ in
  $\bigl(\stabg,\ts\bigr)$, then $u_{n}\to u_{0}$.
\end{cor}

We have a natural action of $G$ on $\Sigmah$ given by
\begin{equation}
  \label{eq:16}
  \gamma\cdot (H,\pi)=(\gamma\cdot H,\gamma\cdot
  \pi)
\end{equation}
where $p_{0}(H)=s(\gamma)$ and $\gamma\cdot \pi(a)=\pi(\gamma^{-1}a\gamma)$.

\begin{lemma}\label{lem-act-sigma}
  The $G$-action on $\Sigmah$ defined by \eqref{eq:16} is continuous.
\end{lemma}
\begin{proof}
  Suppose $\gamma_{n}\to\gamma$ in $G$ and
  $(H_{n},\pi_{n})\to (H,\pi)$ in $\Sigmah$ with
  $p_{0}(H_{n})=s(\gamma_{n})$.  Using \cite{wil:toolkit}*{Lemma~3.22}
  for example, it is clear that
  $\gamma_{n}\cdot H_{n}\to \gamma\cdot H$.  Suppose that
  $a_{n}\in \gamma\cdot H_{n}$ and that $a_{n}\to a_{0}$.  Then
  $\gamma_{n}^{-1}a_{n}\gamma_{n}\in H_{n}$ and converges to
  $\gamma^{-1} a_{0}\gamma$.  Then Proposition~\ref{prop-mrw3.3}
  implies that
  $\gamma_{n}\cdot \pi_{n}(a_{n})=\pi_{n}(\gamma_{n}^{-1}a_{n}
  \gamma_{n})\to \pi(\gamma^{-1}a_{0}\gamma)=\gamma\cdot \pi(a_{0})$.
  Another application of Proposition~\ref{prop-mrw3.3} implies that
  $\gamma_{n}\cdot (H_{n},\pi_{n})\to \gamma\cdot (H,\pi)$ as
  required.
\end{proof}

Since $\gamma\cdot G(s(\gamma))=G(r(\gamma))$, we get a similar action
of $G$ on $\stabg$:
$$\gamma\cdot (s(\gamma),\pi)=(r(\gamma),\gamma\cdot \pi).$$

\begin{cor}
  \label{cor-act-stab} The $G$-action on $\stabg$ is continuous.
\end{cor}
\begin{proof}
  Suppose that $\gamma_{n}\to \gamma$ in $G$ while
  $(s(\gamma_{n}),\pi_{n})\to (s(\gamma),\pi)$ in $\stabg$.  We need
  to show that
  $(r(\gamma_{n}),\gamma_{n}\cdot \pi_{n})\to (r(\gamma),\gamma\cdot
  \pi)$ in $\stabg$.  We can assume that we have already passed to a
  subsequence and relabeled.  Then we can pass to a subsequence,
  relabel, and assume that $G(s(\gamma_{n}))\to H$ and that
  $\bigl(G(s(\gamma_{n})),\pi_{n}\bigr)\to (H,\pi\restr H)$ in
  $\Sigmah$.  But then Lemma~\ref{lem-act-sigma} implies
  $\gamma_{n}\cdot \bigl(G(s(\gamma_{n})),\pi_{n}\bigr) \to
  \bigl(\gamma\cdot H,\gamma\cdot \pi\restr H\bigr)=(\gamma\cdot
  H,(\gamma\cdot \pi)\restr{\gamma\cdot H})$, and the result follows.
\end{proof}

Recall that if $X$ is a topological space, then we write $(X)^{\sim}$
for the ``$T_{0}$-ization'' of $X$ as defined in
\cite{wil:crossed}*{Definition~6.9}.  Thus $(X)^{\sim}=X/\!\!\sim$
where $x\sim y$ if and only if
$\overline{\sset x}=\overline{\sset y}$.  We give $(X)^{\sim}$ the
quotient topology with respect to the natural map $q:X\to (X)^{\sim}$.
This is the largest topology making $q$ continuous.  Since
\begin{align}
  \label{eq:43a}
  \tau:=\set{V\subset (X)^{\sim}:\text{$q^{-1}(V)$ is  open in $X$}}
\end{align}
is a topology on $(X)^{\sim}$ making $q$ continuous, it follows that
$V\subset (X)^{\sim}$ is open if and only if $q^{-1}(V)$ is open in
$X$.  The space $(X)^{\sim}$ is always $T_{0}$, and if $f:X\to Y$ is a
continuous map into a $T_{0}$-space $Y$, then $f$ factors through a
continuous map $f':(X)^{\sim}\to Y$ given by $f'(q(x))=f(x)$
\cite{wil:crossed}*{Lemma~6.10}.

In the sequel, we will focus on the $T_{0}$-izaion, $\stabggt$, of the
orbit space $\stabgg$.  Therefore the next result will be useful.

\begin{lemma}\label{lem-k-open}
  Let $X$ by a $G$-space and let $k:X\to (G\backslash X)^{\sim}$ be
  the natural map.  Then $k$ is continuous and open (cf.,
  \cite{wil:crossed}*{Lemma~6.12}).
\end{lemma}
\begin{proof}
  Since $k$ is the composition of continuous maps, it is clearly
  continuous.  To see that $k$ is open, it suffices to see that
  \begin{align}
    \label{eq:41}
    k^{-1}\bigl(k(U)\bigr)=G\cdot U
  \end{align}
  when $U$ is open in $X$.  Since $k(x)=k(\gamma\cdot x)$, we clearly
  have
  \begin{align}
    \label{eq:42}
    G\cdot U\subset k^{-1}\bigl(k(U)\bigr).
  \end{align}
  Suppose $k(x)\in k(U)$.  Then $k(x)=k(y)$ with $y\in U$.  Then
  $\overline{G\cdot x}=\overline{G\cdot y}$.  Therefore there are
  $(\gamma_{i})\subset G$ such that $\gamma_{i}\cdot x\to y$.  Since
  $U$ is open, we eventually have $\gamma_{i}\cdot x$ in $U$.  But
  then $x=\gamma_{i}^{-1}\cdot (\gamma_{i}\cdot x)$ is eventually in
  $G\cdot U$ and $x\in G\cdot U$.
\end{proof}

\section{$\primcsg$}
\label{sec:primcsg}

In this section we start our examination of $\primcsg$ and the spectrum
$\cs(G)^{\wedge}$ by defining a continuous map of $\stabg$ into
$\cs(G)^{\wedge}$ and examining its properties.

\begin{prop}\label{prop-main-cts}
  Suppose that $G$ is a second countable locally compact Hausdorff
  groupoid with abelian isotropy and with a Haar system.  If
  $(u,\pi)\in\stabg$, then $\irrind(u,\pi)=\indsg\bigl(G(u),\pi\bigr)$ is an
  irreducible represenaton of $\cs(G)$.  Then  $\irrind$ induces a
  continuous map of $\stabg$ into $\cs(G)^{\wedge}$ that is constant
  on $G$-orbits and factors through a continuous map $\irrindq$ of
  $\stabgg$ into $\cs(G)^{\wedge}$.  If
  $\psind(u,\pi)=\ker\bigl(\irrind(u,\pi)\bigr) \bigr)$, then $\psind$
  is continuous as a map of $\stabg$ into $\primcsg$ and factors
  through a continuous map $\psindq$ of $\stabggt$ into $\primcsg$.
  If $G$ is amenable, then $\psind$, and hence $\psindq$, is
  surjective.
\end{prop}

We believe that $\psindq$ is often a homeomorphism and one of our
immediate goals is to establish general conditions for which it is.
But in general, we do not even know if $\psindq$ is always injective.
However, in the GCR case, we can sharpen this result considerably.

\begin{cor}\label{cor-main-cts}
  Suppose that $G$ is as in the proposition, and that $\gugo$ is a
  $T_{0}$ topological space.  Then $\stabgg$ is $T_{0}$ and $\irrind$
  factors through a continuous bijection $\irrindq$ of $\stabgg$ onto
  $\cs(G)^{\wedge}$.
\end{cor}

We need a number of preliminary results before proving the proposition
and its corollary.  Recall that our standing assumptions are that $G$
is second countable, has abelian isotropy, and has a Haar system.

\begin{lemma}[\cite{ionwil:pams08}]
  \label{lem-irr} If $(u,\pi)\in\stabg$, then
  $\indsg\bigl(G(u),\pi\bigr)$ is an irreducible representation of
  $\cs(G)$.
\end{lemma}
\begin{proof}
  If $(u,\pi)\in\stabg$, then $\Ind_{G(u)}^{G}\pi$ is irreducible by
  \cite{ionwil:pams08}*{Theorem~5}.  Hence
  $\indsg\bigl(G(u),\pi\bigr)$ is irreducible by 
  Lemma~\ref{lem-ind-equiv}.
\end{proof}

\begin{lemma}\label{lem-to-amen-gcr}
  If $\gugo$ is $T_{0}$, then $G$ is amenable and $\cs(G)$ is GCR.
\end{lemma}
\begin{proof}
  If $\gugo$ is $T_{0}$, then since the isotropy groups are abelian,
  and hence GCR, $\cs(G)$ is GCR by \cite{cla:iumj07}*{Theorem~1.4}.

  On the other hand, if $\gugo$ is $T_{0}$, then since all the
  isotropy groups are amenable, $G$ is amenable by
  \cite{wil:toolkit}*{Theorem~9.86}.
\end{proof}

\begin{lemma}
  \label{lem-pre-cont} Let $\bigl((u_{n},\pi_{n})\bigr)$ be a sequence
  in $\stabg$ such that $\bigl(G(u_{n}),\pi_{n}\bigr)\to (H,\sigma)$
  in $\Sigmah$.  If $\gamma\in G(u)^{\wedge}$ is such that
  $\gamma\restr H=\sigma$, then
  $\irrind(u_{n},\pi_{n})\to \irrind(u,\gamma)$ in $\cs(G)^{\wedge}$.
\end{lemma}
\begin{proof}
  Since induction is continuous by Lemma~\ref{lem-ind-cts},
  \begin{equation}
    \label{eq:17}
    \ker\bigl(\indsg(G(u_{n}),\pi_{n})\bigr) \to
    \ker\bigl(\indsg(H,\sigma)\bigr)
  \end{equation}
  in $\I(\cs(G))$.

  Since $G(u)$ is abelian, and hence amenable, we can apply
  \cite{gre:jfa69}*{Theorem~5.1} to conclude that as representations
  of the group $G(u)$,
  \begin{equation}
    \label{eq:18}
    \ker\bigl(\Ind_{H}^{G(u)} \sigma\bigr) = \ker\bigl(\Ind_{H}^{G(u)}
    \gamma\restr H\bigr) \subset \ker \gamma.
  \end{equation}
  Therefore we can use induction in stages (cf.,
  \cite{ionwil:pams08}*{Theorem~4}) as well as
  Lemma~\ref{lem-ind-equiv} to conclude that
  \begin{align}
    \label{eq:19}
    \ker\bigl(\indsg(H,\sigma)\bigr)
    &= \ker\bigl(\Ind_{H}^{G}\sigma\bigr)
      =\ker\bigl(\Ind_{G(u)}^{G}\bigl(\Ind_{H}^{G(u)} \sigma
      \bigr)\bigr) \\
    &\subset \ker\bigl(\Ind_{G(u)}^{G}(\gamma)\bigr) = \ker
      \bigl(\indsg(G(u),\gamma)\bigr) =\ker\bigl(\irrind(u,\gamma)\bigr)
      . 
  \end{align}
  Thus, in classical terms, $\irrind(u,\gamma)$ is weakly contained in
  $\indsg(H,\sigma)$.  Now the result follows by untangling
  definitions as in \cite{echeme:em11}*{Lemma~4.7(ii)}.
\end{proof}

\begin{cor}\label{cor-k-cts}
  The map $\psind:\stabg\to \primcsg$ is continuous.
\end{cor}
\begin{proof}
  Suppose that $(u_{n},\pi_{n})\to (u_{0},\pi_{0})$ in $\stabg$.  We
  need to verify that $\psind(u_{n},\pi_{n})\to \psind(u_{0},\pi_{0})$
  in $\primcsg$.  If this fails, then after passing to a subsequence
  and relabeling, there is a $J\in\I(\cs(G))$ such that
  $\psind(u_{0},\pi_{0})\in \O_{J}$, but
  $\psind(u_{n},\pi_{n}) \notin \O_{J}$ for all $n$.

  But then we can pass to another subsequence, relabel, and assume
  that $u_{n}\to u_{0}$, $G(u_{n})\to H \subset G(u_{0})$, and that
  $(G(u_{n}),\pi_{n}) \to (H,\pi_{0}\restr H)$.  By
  Lemma~\ref{lem-pre-cont}, this implies
  $\psind(u_{n},\pi_{n})=\ker\bigl( \irrind(u_{n},\pi_{n})\bigr)\to
  \ker\bigl( \irrind(u_{0},\pi_{0}) \bigr)=\psind(u_{0},\pi_{0})\in
  \O_{J}$.  This leads to a contradiction and completes the proof.
\end{proof}

If $u\in \go$, we let $[u]$ be the orbit $G\cdot u$ in $\go$.

\begin{lemma}
  \label{lem-basic-tools} Let $G$ and $\Sigma$ be as in
  Proposition~\ref{prop-main-cts}.
  \begin{enumerate}
  \item If $\gamma\in G$ and $(H,\sigma)\in\Sigmah$, then
    $\indsg(\gamma\cdot (H,\sigma))$ is equivalent to
    $\indsg(H,\sigma)$.
  \item If $\irrind(u,\pi)$ is equivalent to $\irrind(v,\sigma)$, then
    $[u]=[v]$.
  \item If $\gugo$ is $T_{0}$ and if $\irrind(u,\pi)$ is equivalent to
    $\irrind(u,\sigma)$, then $\pi=\sigma$.
  \end{enumerate}
\end{lemma}
\begin{proof}
  (a) If $H=G(u)$, then this is exactly
  \cite{wil:toolkit}*{Lemma~5.48}.  The general case is proved
  similarly.

  (b) This is \cite{wil:toolkit}*{Proposition~5.50}.

  (c) Let $j:\cs(G)\to \cs(G\restr{\overline{[u]}})$ be the usual map
  \cite{wil:toolkit}*{Theorem~5.1}.  Then $\irrind(u,\pi)$ is
  equivalent to $\Ind_{G(u)}^{G\restr{\overline{[u]}}}\pi\circ j$ by
  \cite{wil:toolkit}*{Corollary~5.29}.  Hence
  $\Ind_{G(u)}^{G\restr{\overline{[u]}}}\pi$ and
  $\Ind_{G(u)}^{G\restr{\overline{[u]}}}\sigma$ are equivalent.  Since
  $\gugo$ is $T_{0}$, $[u]$ is open in $\overline{[u]}$ by the
  Mackey--Glimm--Ramsay Dichcotomy \cite{wil:toolkit}*{Theorem~2.27}.
  Therefore $\Ind_{G(u)}^{G\restr{\overline{[u]}}}\pi$ is the
  canonical extension of $\Ind_{G(u)}^{G\restr{[u]}}\pi$ to
  $\cs(G\restr{\overline{[u]}})$ by \cite{wil:toolkit}*{Lemma~5.31}.
  Therefore $\Ind_{G(u)}^{G\restr{[u]}}\pi$ and
  $\Ind_{G(u)}^{G\restr{[u]}}\sigma$ are equivalent.  Since
  $G\restr{[u]}$ and $G(u)$ are equivalent groupoids,
  $\pi\mapsto \Ind_{G(u)}^{G\restr{[u]}}\pi$ induces a homeomorphism
  of $\cs(G(u))^{\wedge}$ onto $\cs(G\restr{[u]})^{\wedge}$ by
  \cite{rw:morita}*{Corollary~3.33}. The result follows.
\end{proof}


\begin{proof}[Proof of Proposition~\ref{prop-main-cts}]
  We saw that $\indsg\bigl(G(u),\pi\bigr)$ is irreducible in
  Lemma~\ref{lem-irr}.  Since the topology on $\cs(G)^{\wedge}$ is
  pulled back from $\primcsg$, the continuity of $\irrind$ follows
  from Corollary~\ref{cor-k-cts}.  Then $\irrind$ is constant on
  $G$-orbits by Lemma~\ref{lem-basic-tools}(a) and $\irrindq$ is
  well-defined and necessarily continuous (since $\stabgg$ has the
  quotient topology).

  The continuity of $\psind$ is immediate as is the observation that
  it factors through $\stabgg$.  Since $\primcsg$ is $T_{0}$, then
  $\psind$ factors through a continuous map $\psindq$ by the univeral
  property of the $T_{0}$-ization \cite{wil:crossed}*{Lemma~6.10}.

  If $G$ is amenable, then the surjectivity of $\psind$ follows from
  the Effros-Hahn Conjecture for groupoid \cs-algebras
  \cite{ionwil:iumj09}*{Theorem~2.1}.
\end{proof}

\begin{proof}[Proof of Corollary~\ref{cor-main-cts}]
  Since $\gugo$ is $T_{0}$, we have $G$ amenable and $\cs(G)$ GCR by
  Lemma~\ref{lem-to-amen-gcr}.  
  Since $\cs(G)$ is GCR, we can identify $\cs(G)^{\wedge}$ and
  $\primcsg$ by \cite{ped:cs-algebras}*{Theorem~6.1.5}.  Since $G$ is
  also amenable, $\irrindq$ is continuous and surjective by the last
  part of Proposition~\ref{prop-main-cts}.

  It will suffice to see that $\irrindq$ is injective since
  $\cs(G)^{\wedge}$ is $T_{0}$ if $\cs(G)$ is GCR.  But if
  $\irrind(u,\pi)$ is equivalent to $\irrind(v,\sigma)$, we must have
  $[u]=[v]$ by Lemma~\ref{lem-basic-tools}(b).  Then we can assume
  $v=\gamma\cdot u$.  Then $\irrind(u,\pi)$ and
  $\irrind(u,\gamma^{-1}\sigma)$ are equivalent by
  Lemma~\ref{lem-basic-tools}(a).  Since $\gugo$ is $T_{0}$,
  $\gamma^{-1}\cdot \sigma=\pi$ by Lemma~\ref{lem-basic-tools}(c).
  But then $G\cdot (u,\pi)=G\cdot (v,\sigma)$ and $\irrindq$ is
  injective as required.
\end{proof}

\section{Renault's Simplicity Result}
\label{sec:rena-simpl-result}

As something of an aside, we want to see that
Proposition~\ref{prop-main-cts} already implies a nice simplicity
result in our setting for amenable $G$.  It is a
special case of Renault's fundamental
\cite{ren:jot91}*{Corollary~4.9}.

First we need to recall some terminology.  In
\cite{ren:jot91}*{Definition~4.1}, Renault says that the isotropy of
$G$ is \emph{discretely trivial} at $v\in\go$ if for each compact set
$K$ in $G$ there is a neighborhood $V$ of $v$ in $\go$ such that
$u\in V$ implies $G(u)\cap K\subset\sset u$.  Of course this also
implies $G(v)=\sset v$.  This notion is of most import when $G$ is
\'etale in which case this isotropy is discretely trivial at $v$ if
and only if $G(v)=\sset v$.

\begin{prop}
  \label{prop-simple} Suppose that $G$ is an \emph{amenable} second
  countable locally compact Hausdorff groupoid with abelian isotropy
  and a Haar system. Suppose that the action of $G$ on $\go$ is
  minimal; that is, we assume $[u]$ is dense in $\go$ for all
  $u\in\go$.  If in addition there is a $u_{0}\in\go$ such that
  the isotropy is discretely trivial at $u_{0}$, then $\cs(G)$ is
  simple.
\end{prop}

\begin{proof}
  In view Proposition~\ref{prop-main-cts}, it will suffice to
  see that $\stabggt$ is reduced to a single point.  Suppose
  $(u,\pi)\in\stabg$.  Since the orbit $[u_{0}]$ is dense, there are
  $\gamma_{n}\in G$ such that $\gamma_{n}\cdot u_{0}\to u$.  But
  $G(\gamma_{n}\cdot u_{0})=\sset{\gamma_{n}\cdot u_{0}}$ converges to
  $H=\sset u$ in $\so$.  Then it is not hard to see that
  $\bigl(G(\gamma_{n}\cdot u_{0}),1\bigr)=\gamma_{n}\cdot
  \bigl(G(u_{0}),1\bigr)$ converges to $(H,1)$ in $\Sigmah$.  It
  follows from Proposition~\ref{prop-conv-ts} that
  $\gamma_{n}\cdot (u_{0},1) \to (u,\pi)$ in $\stabg$.  Therefore
  \begin{equation}
    \label{eq:11}
    (u,\pi)\in \overline{G\cdot (u_{0},1)}.
  \end{equation}

  On the other hand, $[u]$ is dense and there are $\eta_{n}\in G$ such
  that $\eta_{n}\cdot u\to u_{0}$ in $\go$.  Since
  $G(u_{0})=\sset{u_{0}}$ it follows that
  $G(\eta_{n}\cdot u)\to G(u_{0})$ in $\so$
  \cite{wil:toolkit}*{Lemma~6.10}.  We claim that
  $\bigl(G(\eta_{n}\cdot u),\eta_{n}\cdot \pi\bigr)$ converges to
  $\bigl(G(u_{0}),1\bigr)$ in $\Sigmah$.  We will apply 
  Proposition~\ref{prop-mrw3.3}.  Suppose that
  $a_{n}\in G(\eta_{n}\cdot u)\to u_{0}$.  Then
  $K=\set{\eta_{n}\cdot u}\cup \sset{u_{0}}$ is compact.  Since the
  isotropy is discretely trivial at $u_{0}$, we eventually have
  $a_{n}=\eta_{n}\cdot u$.  Then $\eta_{n}\cdot \pi(a_{n})=1$ and the
  claim follows.  Therefore $\eta_{n}\cdot (u,\pi)$ converges to
  $(u_{0},1)$ in $\stabg$.  Then
  \begin{equation}
    \label{eq:12}
    (u_{0},1)\in \overline{G\cdot (u,\pi)}.
  \end{equation}

  Since $(u,\pi)\in\stabg$ is arbitrary, $\stabggt$ is a single point
  and the result is proved.
\end{proof}

\section{Subgroup Sections}
\label{sec:subgroup-sections}

\begin{definition}
  \label{def-sub-section} A continuous map $\sg:\go\to\so$ is called a
  \emph{subgroup section} if $p_{0}(\sg(u))=u$ for all $u\in\go$.  We
  call $\sg$ \emph{equivariant} if
  $\sg(\gamma\cdot u)=\gamma\cdot \sg(u)$.
\end{definition}

\begin{remark}
  In the extreme case that $u\mapsto G(u)$ is continuous,
  then $\sg(u)=G(u)$ is an equivariant subgroup section.  At the other
  extreme, $\sg(u)=\sset u$ for all $u\in\go$ is always an equivariant
  subgroup section. 
\end{remark}

Given a subgroup section $\sg$, we define a left action of
$\cc(\Sigma)$ on $\cc(G)$ by
\begin{equation}
  \label{eq:20}
  b\cdot f(\gamma)=\int_{\sgr}b\bigl(\sgr,t\bigr)f(t^{-1}\gamma)
  \,d\beta^{\sgr} (t).
\end{equation}
It is not hard to check that if $b\in\cc(\Sigma)$ and $f\in\cc(G)$,
then $b\cdot f\in \cc(G)$ (for example, see
\cite{wil:toolkit}*{Lemma~3.29}).

In the sequel, we will realize the multiplier algebra $M(\cs(G))$ as
the algebra of adjointable operators $\L(\cs(G))$ where $\cs(G)$ is
viewed a right Hilbert module over itself (see
\cite{rw:morita}*{\S2.3}).
\begin{prop}
  \label{prop-P-map} Suppose that $\sg$ is a subgroup section.  Then
  there is a homomorphism $P_{\sg}:\cs(\Sigma)\to M(\cs(G))$ such that
  for $b\in\cc(\Sigma)$ and $f\in \cc(G)$ we have
  $P_{\sg}(b)f=b\cdot f$ where the latter is given by \eqref{eq:20}.
\end{prop}
\begin{remark}
  When the choice of $\sg$ is clear, we simply write $P$ in place of
  $P_{\sg}$.  In the case $u\mapsto G(u)$ is continuous and
  $\sg(u)=G(u)$, this result is a special case of
  \cite{goe:rmjm12}*{Proposition~2.19}.
\end{remark}

\begin{proof}[Sketch of the Proof]  
  If $f,g\in\cc(G)$, then we let $\Rip<f,g>=f^{*}*g$.  Then we can use
  Lemma~\ref{lem-3.25} and the observation that
  $\omega(\eta t,H)=\omega(\eta,H)$ if $t\in H$ to conclude that for
  all $b\in \cc(\Sigma)$ we have
  \begin{equation}
    \label{eq:23}
    \Rip<P(b)f,g>=\Rip<f,P(b^{*})g>.
  \end{equation}
  Similar calculations show that $P(b*b')f=P(b)\bigl(P(b')f\bigr)$ and
  that $b\mapsto \Rip<P(b)f,g>$ is continuous when $\cc(G)$ is given
  the inductive limit topology.  Then we can use
  \cite{kmqw:nyjm10}*{Proposition~1.7} to see that $P$ extends to a
  homomorphism of $\cs(\Sigma)$ into $\L(\cs(G))=M(\cs(G))$.  (We can
  dispense with condition~(iii) of
  \cite{kmqw:nyjm10}*{Proposition~1.7} using
  \cite{wil:toolkit}*{Corollary~8.4}.)
\end{proof}

\begin{cor}\label{cor-pstar-cts}
  Let $\sg$ be a subgroup section and let $P=P_{\sg}$ be the
  corresponding map.  Then there is a continuous map
  $P^{*}:\I(\cs(G))\to\I(C_{0}(\Sigmah))$ given by
  \begin{equation}
    \label{eq:24}
    P^{*}(J)=\set{\hat b\in C_{0}(\Sigmah):P(b)\cdot \cs(G)\subset J}.
  \end{equation}
\end{cor}
\begin{proof}
  This is \cite{wil:crossed}*{Lemma~8.35} combined with the Gelfand
  isomorphism of $\cs(\Sigma)$ with $C_{0}(\Sigmah)$.
\end{proof}

\begin{prop}
  \label{prop-pstar} Suppose that $\sg$ is a subgroup section and that
  $P^{*}:\I(\cs(G))\to\I(C_{0}(\Sigmah))$ is the associated map.  If
  $J(u,\pi)=\ker(\irrind(u,\pi))$, then $P^{*}(J(u,\pi))$ is the ideal
  of functions in $C_{0}(\Sigmah)$ that vanish on the closure of
  \begin{equation}
    \label{eq:25}
    \set{\bigl(\sg(\gamma\cdot u),(\gamma\cdot \pi)\restr{\sg(\gamma\cdot
        u)}\bigr)\in \Sigmah:\gamma\in G_{u}}
  \end{equation}
  in $\Sigmah$.
\end{prop}
\begin{remark}\label{rem-pstar}
  If $\sg$ is equivariant then
  $(\gamma\cdot\pi)\restr{\sg(\gamma\cdot u)}=\gamma\cdot
  (\pi\restr{\sgu})$ and $P^{*}(J(u,\pi))$ is the ideal of functions
  vanishing on the orbit closure
  $\overline{G\cdot (\sgu,\pi\restr{\sgu})}$.
\end{remark}

We will need some preliminaries before embarking on the proof of the
proposition.  Let $L=\Ind_{G(u)}^{G}\pi$.  Since $L$ is equivalent to
$\irrind(u,\pi)$, $\ker L=J(u,\pi)$.  Recall that $L$ acts on $\H$
which is the completion of $\cc(G)$ with respect to the inner product
$\ipu(\cdot | \cdot)$ given in \eqref{eq:5} by taking $H=G(u)$.  (There is
no harm in substituting $\cc(G)$ for $\cc(G_{u})$ here.)  We let $\uL$
be the representation of $\cs(\Sigma)$ on $\H$ given by
$\bar L\circ P$ where $\bar L$ is the canonical extension of $L$ to
$M(\cs(G))$.

\begin{lemma}
  Let $b\in\cs(\Sigma)$.  Then $\hat b\in P^{*}(J(u,\pi))$ if and only
  if
  \begin{equation}
    \label{eq:26}
    \bip(\uL(b)f|g)_{u}=0\quad\text{for all $f,g\in \cc(G)$.}
  \end{equation}
\end{lemma}
\begin{proof}
  We have $\hat b\in P^{*}(J(u,\pi))$ if and only if
  \begin{equation}
    \label{eq:27}
    \bip(\uL(b)L(e)f| g)_{u}=\bip(L(P(b)e) f| g)_{u}=0\quad\text{for
      all $e,f,g\in\cc(G)$.} 
  \end{equation}
  Since $\bip(\uL(b)L(e)f| g)_{u}=\bip(\uL(b)(e*f)|g)_{u}$, it follows
  that \eqref{eq:26} implies $\hat b\in P^{*}(J(u,\pi))$.

  On the other hand, if $(e_{i})$ is an approximate identity for
  $\cc(G)$ in the inductive limit topology
  (\cite{wil:toolkit}*{Proposition~1.49}), then
  \begin{align}
    \label{eq:28}
    \bip(\uL(b)f|g)_{u}
    &=\bip(f| \uL(b)^{*}g)_{u}=\lim_{i}\bip(L(e_{i})f|\uL(b)^{*}g)_{u}=
      \lim_{i} \bip(e_{i}*f| \uL(b)^{*}g)_{u} \\
    &= \bip(\uL(b)(e_{i}*f)|g)_{u}.
  \end{align}
  Therefore if $\hat b\in P^{*}(J(u,\pi))$, we have
  $\bip(\uL(b)f|g)_{u}=0$ for all $f,g\in\cc(G)$.
\end{proof}

We need a technical observation.  As in the previous proof, we employ
an approximate identity $(e_{i})$ for $\cc(G)$ in the inductive limit
topology.

\begin{lemma}
  \label{lem-technical} If $b\in\cc(\Sigma)$, then
  \begin{equation}
    \label{eq:29}
    \ipu(\uL(b)f|g)=\ipu(b\cdot f|g).
  \end{equation}
  (This is equivalent to saying that $\uL(b)f=b\cdot f$ as elements of
  the Hilbert space $\H$.)
\end{lemma}
\begin{proof}
  Since $P(b)\in \L(\cs(G)_{\cs(G)})$, it respects the right action of
  $\cs(G)$.  Thus if $e,f\in\cc(G)$, $P(b)(e*f)=(P(b)e)*f$.  Thus in
  $\cs(G)$, $b\cdot (e*f)=(b\cdot e)*f$.  Since the universal norm is
  a norm on $\cc(G)$, this means $b\cdot (e*f)=(b\cdot e)*f$ in
  $\cc(G)$.  (This may also be verified directly.)  Then
  \begin{align}
    \label{eq:30}
    \ipu(\uL(b)f|g)
    &=\lim_{i}\ipu(\uL(b)L(e_{i})f|g)
      =\lim_{i}\ipu(L(P(b)e_{i})f|g)
      =\lim_{i}\ipu((b\cdot e_{i})*f|g)\\
    &=\lim_{i}\ipu(b\cdot (e_{i}*f)|g)
      =\lim_{i}\ipu(b\cdot f|g)
  \end{align}
  since $b\cdot (e_{i}*f)\to b\cdot f$ in the inductive limit
  topology.
\end{proof}

\begin{proof}[Proof of Proposition~\ref{prop-pstar}]
  Using Lemma~\ref{lem-technical}, observe that if $b\in\cc(\Sigma)$,
  then
  \begin{align}
    \ipu(\uL(b)f|g) = \ip(b\cdot f|g)_{u} \hskip-4cm
    &\\
    &=\int_{G(u)}\int_{G}\overline{g(\gamma)} b\cdot f(\gamma t)\pi(t)
      \,d\lambda_{u}(\gamma) \,d\beta^{G(u)}(t) \\
    &= \int_{G}\int_{G(u)}\int_{\sgr} \overline{g(\gamma)}
      b\bigl(\sgr,s\bigr) f(s^{-1}\gamma t) \pi(t)
    \\
    &\hskip3in d\beta^{\sgr}(s) \,d\beta^{G(u)}(t)
      \,d\lambda_{u}(\gamma) \\
    \intertext{which, after $t\mapsto \gamma^{-1} t \gamma$, is}
    &=\int_{G}\omega(\gamma,G(u)) \int_{G(r(\gamma))} \int_{\sgr}
      \overline{g(\gamma)} b\bigl(\sgr,s\bigr) f(s^{-1}t \gamma)
      \pi(\gamma^{-1}t \gamma) \\
    &\hskip3in d\beta^{\sgr}(s) \,d
      \beta^{G(r(\gamma))}(t) \,d\lambda_{u}(\gamma)\\
    \intertext{which, after Fubini and $t\mapsto st$, is}
    &= \int_{G}\omega(\gamma,G(u)) \int_{\sgr}
      \int_{G(r(\gamma))} \overline{g(\gamma)}
      b\bigl(\sgr,s\bigr) f(t\gamma) \pi(\gamma^{-1}st\gamma) \\
    &\hskip3in d\beta^{G(r(\gamma))}(t) \,d\beta^{\sgr}(s) \,d
      \lambda_{u}(\gamma) \\
    &= \int_{G} \omega(\gamma,G(u)) \overline{g(\gamma)}\Bigl(
      \int_{\sgr} 
      b\bigl(\sgr,s\bigr) \pi(\gamma^{-1}s\gamma)
      \,d\beta^{\sgr}(s) \Bigr) \\
    &\hskip 1in \int_{G(r(\gamma))}
      f(t\gamma)\pi(\gamma^{-1}t\gamma)\,d\beta^{G(r(\gamma))} (t)
      \,d\lambda_{u}(\gamma) \\
    &= \int_{G} \omega(\gamma,G(u)) \overline{g(\gamma)}\hat
      b\bigl(\sgr,(\gamma\cdot \pi)\restr{\sgr}\bigr)
      \int_{G(r(\gamma))}
      f(t\gamma)\pi(\gamma^{-1}t\gamma)\\
    &\hskip3in d\beta^{G(r(\gamma)} (t)
      \,d\lambda_{u}(\gamma) \\
    &=\int_{G}\int_{G(u)} \hat b\bigl(\sgr,(\gamma\cdot \pi)
      \restr{\sgr})) 
      \overline{g(\gamma) }f(\gamma 
      t)\pi(t) 
      \,d\beta^{G(u)}(t) \,d\lambda_{u}(\gamma) \\
    &=\int_{G} \hat b\bigl(\sgr, (\gamma\cdot
      \pi) \restr{\sgr})\bigr) \overline{g(\gamma) } \hat 
      f(\gamma,\pi) \,d\lambda_{u}(\gamma)\label{eq:38}
  \end{align}
  where
  \begin{align}
    \label{eq:37}
    \hat f(\gamma,\pi):=\int_{G(u)} f(\gamma t)\pi(t)\,d\beta^{G(u)}(t).
  \end{align}

  Note that $\hat f$ is continuous on $G_{u}$.  Hence
  $\gamma \mapsto g(\gamma)\hat f(\gamma,\pi)$ is in $\cc(G_{u})$.  If
  $b\in\cs(\Sigma)$, then there is a sequence $(b_{i})$ in
  $\cc(\Sigma)$ converging to $b\in\cs(\Sigma)$ and bounded in norm in
  $\cs(\Sigma)$.  Then $(\hat b_{i})$ is bounded in $C_{0}(\Sigmah)$
  and $\hat b_{i}\to \hat b$ uniformly.  Thus
  \begin{multline}
    \label{eq:57}
    \lim_{i} \int_{G}\hat b_{i}\bigl(\sgr, (\gamma\cdot
    \pi)\restr{\sgr})\bigr) \overline{g(\gamma)} \hat
    f(\gamma,\pi)\,d\lambda_{u}(\gamma)\\
    =\int_{G} \hat b(\sgr,(\gamma\cdot \pi)\restr{\sgr}))
    \overline{g(\gamma) } \hat f(\gamma,\pi) \,d\lambda_{u}(\gamma)
  \end{multline}
  by the dominated converge theorem.  On the other hand
  \begin{align}
    \label{eq:58}
    \lim_{i}\bip(\uL(b_{i})f|g)_{u} = \bip(\uL(b)f|g)_{u}.
  \end{align}
  It follows that $b\in\cs(\Sigma)$, then
  \begin{equation}
    \label{eq:31}
    \ipu(\uL(b)f|g)=\int_{G} \hat b\bigl(\sgr, (\gamma\cdot
    \pi) \restr{\sgr})\bigr) \overline{g(\gamma) } \hat 
    f(\gamma,\pi) \,d\lambda_{u}(\gamma)
  \end{equation}
  for all $f,g\in\cc(G)$.  Thus if \eqref{eq:25} holds, then
  $\hat b\in P^{*}(J)$.

  Suppose that $b\in\cs(\Sigma)$ is such that
  $\phi(\gamma)=\hat b\bigl(\sgr,(\gamma\cdot \pi)\restr{\sgr})\bigr)$
  is non-zero at $\gamma_{0}\in G_{u}$.  Then there is a $h\in \cc(G)$
  such that
  \begin{align}
    \label{eq:60}
    \phi(\gamma)h(\gamma)\ge0 \quad\text{if $\gamma\in G_{u}$}
  \end{align}
  and such that $\phi(\gamma_{0})h(\gamma_{0})=1$.  There is a
  $\zeta\in \cc(G(u))$ such that
  \begin{align}
    \label{eq:61}
    \int_{G(u)} \zeta(t)\pi(t)\,d\beta^{G(u)}(t)=1.
  \end{align}
  Then there is a $f\in \cc(G)$ such that $f(\gamma_{0}t)=\zeta(t)$
  for all $t\in G(u)$.  Thus $\hat f(\gamma_{0},\pi)=1$.  Since
  $\hat f(\gamma,\pi)$ is non-zero near $\gamma_{0}$, we can find
  $k\in \cc(G)$ such that
  \begin{align}
    \label{eq:62}
    k(\gamma)\hat f(\gamma,\pi)\ge0\quad\text{if $\gamma\in
    G_{u}$, and}
  \end{align}
  such that $k(\gamma_{0})\hat f(\gamma_{0},\pi)=1$.  Let
  $g(\gamma)=\overline{h(\gamma)k(\gamma)}$ and let $f$ be as above.
  Since $\beta^{G(u)}$ has full support, it follows that
  \begin{align}
    \label{eq:63}
    \bip(\uL(b)f|g)_{u}>0.
  \end{align}
  Consequently, $\hat b\notin P^{*}(J)$.  Since
  $( \sgr, (\gamma\cdot \pi)\restr{\sgr}) = (\sg(\gamma\cdot u),
  (\gamma\cdot \pi)\restr{\sg(\gamma\cdot u)})$, this completes the
  proof.
\end{proof}

The following corollary will be useful in the sequel.

\begin{cor}
  \label{cor-orbits} As above, let $J(u,\pi)=\ker\bigl(\irrind(u,\pi)\bigr)$.
  If $J(u_{n},\pi_{n})\to J(u,\pi)$ in $\primcsg$, then after passing
  to a subsequence and relabeling, there are $\gamma_{n}\in G$ such
  that $\gamma_{n}\cdot u_{n}\to u$.
\end{cor}
\begin{proof}
  There is always a trivial (equivariant) subsgroup section
  $\sg:\go\to\so$ given by $\sg(u)=\sset u$.  Then $P^{*}(J(u,\pi))$
  is the ideal of functions in $C_{0}(\Sigmah)$ vanishing on
  $\overline{G\cdot (\sset u,1)}$ in $\Sigmah$.  Since
  $P^{*}(J(u_{n},\pi_{n}))\to P^{*}(J(u,\pi))$ in
  $\I(C_{0}(\Sigmah))$, we can use Lemma~\ref{lem-fix-8.38} to pass to
  a subsequence, relabel, and find
  $a_{n}\in \overline{G\cdot (\sset{u_{n}},1)}$ converging to
  $(\sset u,1)$ in $\Sigmah$.  Since $\Sigmah$ is a metric space, we
  can take $a_{n}\in G\cdot (\sset{u_{n}},1)$.  The result now follows
  easily from Proposition~\ref{prop-mrw3.3}.
\end{proof}

\begin{cor}\label{cor-orbit-closure}
  Suppose that
  $\ker\bigl(\Ind(u,\pi)\bigr)=\ker\bigl(\Ind(v,\sigma)\bigr)$ Then
  $\overline{[u]}=\overline{[v]}$.
\end{cor}
\begin{proof}
  By Corollary~\ref{cor-orbits}, we may as well assume that there are
  $\gamma_{n}\in G$ such that $\gamma_{n}\cdot u\to v$.  Thus
  $v\in \overline{[u]}$.  The result follows by symmetry.
\end{proof}

\section{Continuous Isotropy}
\label{sec:continuous-isotropy}

In this section, we want to consider examples were the isotropy map
$u\mapsto G(u)$ continuous except at isolated points.
Specifically, we make the following definition.

\begin{definition}\label{def-jmp-dis}
  We say that $G$ has continuous isotropy except for jump
  discontinuities if $G$ has an equivariant subgroup section
  $\sg:\go\to\so$ such that the set $D=\set{u\in\go:\sg(u)\not=G(u)}$
  is discrete.
\end{definition}

\begin{remark}\label{rem-free-off-d}
  Let $G$ be such that $G(u)=\sset u$ for all $u$ off a discrete set
  $D\subset \go$.  Then $G$ has continuous isotropy except for jump
  discontinuities.  For a specific example of this, see
  Example~\ref{ex-sink}.  Other more subtle examples of jump
  discontinuities are given in
  Examples~\ref{ex-aidan} and~\ref{ex-combo}.
\end{remark}

\begin{thm}
  \label{thm-jmp-dis} Suppose that $G$ is a second countable locally
  compact Hausdorff groupoid with a Haar system and abelian isotropy.
  Suppose that the isotropy is continuous except for jump
  discontinuities.  If $G$ is amenable, then
  $(u,\pi)\mapsto \psind(u,\pi):=\ker\bigl(\irrind(u,\pi)\bigr)$ is an
  open surjection inducing a homeomorphism $\psindq$ of $\stabggt$
  onto $\primcsg$.  If $\gugo$ is $T_{0}$, then $G$ is amenable and
  $\cs(G)$ is GCR.  Furthermore, in that case,
  $(u,\pi)\mapsto \irrind(u,\pi)$ induces a homeomorphism $\irrindq$
  of $\stabgg$ onto $\cs(G)^{\wedge}$.
\end{thm}

\begin{remark}
  [Goehle] \label{rem-goehle} If the isotropy map $u\mapsto G(u)$ is
  continuous, then Theorem~\ref{thm-jmp-dis} applies.  This allows us
  to recover Goehle's \cite{goe:rmjm12}*{Theorem~3.5} describing
  $\primcsg$ as a special case.  At the same time we can give a
  cleaner description of the topology of $\stabggt$ using results from
  Section~\ref{sec:stabg}.  Specifically, consider
  Lemma~\ref{lem-nice-top} and Remark~\ref{rem-nice-top}.
\end{remark}

For the proof, it will be useful to note that the equivariance of the
subgroup section easily implies the following and we omit the proof.

\begin{lemma}
  \label{lem-closed-orbits} If $G$ has continuous isotropy except for
  jump discontinuities and if $\sg:\go \to \so$ is an equivariant
  subgroup section such that $D=\set{u\in\go:\sg(u)\not=G(u)}$ is
  discrete, then $D$ is $G$-invariant and $[u]$ is closed for all
  $u\in D$.
\end{lemma}

\begin{proof}[Proof of Theorem~\ref{thm-jmp-dis}]
  Let $\sg:\go\to \so$ be a subgroup section such that
  $D=\set{u\in\go:\sg(u)\not=G(u)}$ is discrete.  Let
  $P^{*}:\I(\cs(G))\to \I(C_{0}(\Sigmah))$ be the corresponding
  continuous map from Corollary~\ref{cor-pstar-cts}.
  
  By Proposition~\ref{prop-main-cts}, we know that $\psind$ and
  $\psindq$ are continuous surjections.

  Suppose that $\psind(u,\pi)=\psind(v,\sigma)$.  Then $[v]$ and $[u]$
  have the same closure by Corollary~\ref{cor-orbit-closure}.  Since
  $u\in D$ implies $[u]$ is closed, $u\in D$ implies that $v\in [u]$.
  Thus we may as well assume $v=u$.  Since $G\restr{[u]}$ and $G(u)$
  are equivalent groupoids and since $G(u)$ is abelian,
  $\cs(G\restr{[u]})$ is GCR. Since $\Ind(u,\pi)$ and $\Ind(u,\sigma)$
  factor through $\cs(G\restr{[u]})$ and have the same kernel, it
  follows that $\Ind_{G(u)}^{G\restr{[u]}}\pi$ and
  $\Ind_{G(u)}^{G\restr{[u]}}\sigma$ are equivalent
  \cite{ped:cs-algebras}*{Theorem~6.1.5}.  Since $G\restr{[u]}$ is
  equivalent to $G(u)$, $\pi=\sigma$.  Thus $G\cdot (u,\pi)$ and
  $G\cdot (v,\sigma)$ are equal.  In particular, $(u,\pi)$ and
  $(v,\sigma)$ define the same class in $\stabggt$.

  On the other hand, if $u\notin D$, then we can use
  Proposition~\ref{prop-pstar} and Remark~\ref{rem-pstar} to conclude
  that $P^{*}(\psind(u,\pi))$ is the ideal of functions vanishing on
  $\overline{G\cdot (G(u),\pi)}\subset \Sigmah$.  Thus if
  $\psind(u,\pi)=\psind(v,\sigma)$ we also have $v\notin D$ and
  $\overline{G\cdot (u,\pi)}=\overline{G\cdot (v,\sigma)}$.  Thus
  $(u,\pi)$ and $(v,\sigma)$ define the same class in $\stabggt$ in
  this case as well.  This shows that $\psindq$ is injective.

  In view of Lemma~\ref{lem-k-open}, to see that $\psindq$ is a
  homeomorphism, it will suffice to see that $\psind$ is open. Let $V$
  be a neighborhood of $(u,\pi)$ in $\stabg$.  Suppose that
  $\psind(u_{n},\pi_{n})\to \psind(u,\pi)$. Since this sequence is
  arbitrary, it will suffice to see that $\psind(u_{n},\pi_{n})$ has a
  subsequence which is eventually in $\psind(V)$.

  Suppose that $u\notin D$.  Then we may as well assume that
  $u_{n}\notin D$ for all $n$. However, Corollary~\ref{cor-pstar-cts}
  implies that $P^{*}(\kappa(u_{n},\pi_{n}))\to P^{*}(\kappa(u,\pi))$
  in $\I(C_{0}(\Sigmah))$.  By Lemma~\ref{lem-fix-8.38}, we can pass
  to a subsequence, relabel, and assume that there there are
  $a_{n}\in \overline{G\cdot (G(u_{n}),\pi_{n})}$ converging to
  $(G(u),\pi)$ in $\Sigmah$.  Since $\Sigmah$ is a metric space, there
  are $\gamma_{n}\in G$ such that
  $\gamma_{n}\cdot (G(u_{n}),\pi_{n})\to (G(u),\pi)$ in $\Sigmah$.  Of
  course, we also have $\gamma_{n}\cdot (u_{n},\pi_{n})\to (u,\pi)$ in
  $\stabg$ and $\gamma_{n}\cdot (u_{n},\pi_{n})$ is eventually in $V$.
  Since
  $\psind\bigl(\gamma_{n}\cdot
  (u_{n},\pi_{n})\bigr)=\psind(u_{n},\pi_{n})$ by
  Lemma~\ref{lem-basic-tools}(b), we have shown that
  $\psind(u_{n},\pi_{n})$ is eventually in $\psind(V)$ as required.

  Now assume $u\in D$.  If there are infinitely many $u_{n}$ such that
  $u_{n }\notin D$, then we can pass to a subsequence, relabel, and
  assume that $u_{n}\notin D$ for all $n$.  Just as above,
  $P^{*}(\kappa(u_{n},\pi_{n}))\to P^{*}(\kappa(u,\pi))$ in
  $\I(C_{0}(\Sigmah))$ except now $P^{*}(\kappa(u,\pi))$ is the ideal
  of functions vanishing on
  $\overline{G\cdot (\sg(u),\pi\restr{\sg(u)})}$.  Then we can pass to
  a subsequence, relabel, and assume that there are $\gamma_{n}\in G$
  such that
  $\gamma_{n }\cdot (G(u_{n }),\pi_{n})\to (\sg(u),\pi\restr{\sg(u)})$
  in $\Sigmah$.  But then
  $\gamma_{n}\cdot (u_{n },\pi_{n})\to (u,\pi)$ in $\stabg$ by
  Lemma~\ref{lem-pre-cont}.  Then we have
  $\psind(u_{n},\pi_{n})=\psind(\gamma_{n}\cdot (u_{n},\pi_{n}))$
  eventually in $\psind(V)$.

  This leaves the case where there are infinitely many $u_{n}\in D$.
  Then we can reduce to the case where $u_{n}\in D$ for all $n$.  By
  Corollary~\ref{cor-orbits}, we can assume that there are
  $\gamma_{n}$ such that $\gamma_{n}\cdot u_{n}\to u$.  After
  replacing $(u_{n},\pi_{n})$ with $\gamma_{n}\cdot (u_{n},\pi_{n})$,
  we can assume $u_{n}\ \to u$.  Since $D$ is discrete, we can pass to
  a subsequence, relabel, and assume $u_{n}=u$ for all $n$.  But then
  we can assume
  $\Ind_{G(u)}^{G\restr{[u]}}\pi_{n} \to
  \Ind_{G(u)}^{G\restr{{[u]}}}\pi$ in $\cs(G\restr{[u]})^{\wedge}$.
  Since $G\restr{[u]}$ is equivalent to $G(u)$ this implies
  $\pi_{n}\to \pi$ in $G(u)^{\wedge}$.  Then
  $(u_{n},\pi_{n})=(u,\pi_{n})\to (u,\pi)$ in $\stabg$.

  This completes the proof that $\psind$ is open.

  If $\gugo$ is $T_{0}$, the rest follows from
  Corollary~\ref{cor-k-cts}.
\end{proof}

If the isotropy varies continuously, then we can sharpen our
description of the topology on $\stabg$.

\begin{lemma}\label{lem-nice-top}
  Let $\Sigmac=\set{(G(u),\pi)\in\Sigmah:u\in\go}$.  Then $\Sigmac$ is
  a closed subset of $\Sigmah$ and $(G(u),\pi)\mapsto (u,\pi)$ is a
  homeomorphism of $\Sigmac$ onto $\stabg$.  In particular, a sequence
  $\bigl((u_{n},\pi_{n})\bigr)$ converges to $(u,\pi)$ in $\stabg$ if
  and only if $u_{n}\to u$ and given $a_{n}\in G(u_{n})$ converging to
  $a\in G(u)$ we have $\pi_{n}(a_{n})\to \pi(a)$.
\end{lemma}

\begin{remark}
  \label{rem-nice-top}
  In the continuously varying isotropy case in particular, the
  openness of the map of $\Sigmac$ onto $\stabggt$, and the proof of
  Theorem~\ref{thm-jmp-dis}, give us a rather robust description of
  the topology on $\primcsg$.  If $[u,\pi]$ is the class of $(u,\pi)$
  in $\stabggt$, then we note that a sequence
  $\bigl([u_{n},\pi_{n}]\bigr)$ converges to $[u,\pi]$ if and only if
  every subsequence has a subsequence
  $\bigl([u_{n_{k}},\pi_{n_{k}}]\bigr)$ such there are
  $\gamma_{k}\in G$ such that
  $\gamma_{k}\cdot \bigl(G(u_{n_{k}}),\pi_{n_{k}}\bigr)$ converges to
  $(G(u),\pi)$ in $\Sigmac$.
\end{remark}

\section{Action Groupoids for Abelian Groups}
\label{sec:abelian-group-action}

We pause to see what the known results for action groupoids imply
about our more general constructions.  Recall that if $(\grg,X)$ is a
second countable locally compact transformation group, then the
corresponding action groupoid, or transformation groupoid, is
$G=G(\grg,X)=\set{(x,g,y)\in X\times \grg\times X:x=g\cdot y}$
\cite{wil:toolkit}*{Example~1.12}.  Then $\cs(G)$ can be identified
with the transformation group \cs-algebra $C_{0}(X)\rtimes_{\lt}\grg$
\cite{wil:toolkit}*{Example~1.55}.  If $\grg$ is abelian, then we can
use \cite{wil:tams81}*{Theorem~5.3} to describe $\primcsg$ and its
topology.  However this description is very different than what we
propose here.  Nevertheless, we can use this earlier work to verify our map
$\psindq$ from Section~\ref{sec:primcsg} is a homeomorphism in this
case.

More generally, provided that $\grg$ is amenable or that
$\grg\backslash X$ is $T_{0}$, we can also deal with the case where
there is an abelian subgroup $\grh$ of $\grg$ such that each stability
subgroup, $S_{x}=\set{g\in\grg:g\cdot x=x}$, is contained in $\grh$.
Note that $G=G(\grg,X)$ is amenable if $\grg$ is, or if
$G\backslash X=\grg\backslash X$ is $T_{0}$.  Furthermore, replacing
$\grh$ with $\bigcap_{g\in \grg} g\grh g^{-1}$ allows us to assume
that $\grh$ is normal in $\grg$.  Hence we can apply
\cite{wil:tams81}*{Theorem~5.8} and its corollaries.

It will be helpful to translate the results in \cite{wil:tams81} to
our setting.  As above, we write $S_{x}$ for the stability subgroup
$\set{h\in\grg:h\cdot x=x}$ in $\grg$ to distinguish it from the
isotropy group $G(x)=\set{(x,h,x)\in G:h\in S_{x}}$ in $G$.  However,
if $\pi\in \hat S_{x}$, then we will use the same symbol for the
corresponding character on $G(x)$.  It is not hard to see that the
induced representation $\Ind_{(x,S_{x})}^{\grg}(\pi)$ of the
transformation group \cs-algebra, defined in
\cite{wil:tams81}*{Definition~3.4}, when viewed as a representation of
$\cs(G)$, is equivalent to our $\Ind_{G(x)}^{G}(\pi)$.  Thus if
$(x,\rho)\in X\times\hath$ and if we define
$\psi(x,\rho)=
\ker\bigl(\Ind_{(x,S_{x})}^{\grg}(\rho\restr{S_{x}})\bigr)$ as in
\cite{wil:tams81}, then $\psi(x,\rho)=\psind(x,\rho\restr{G(x)})$.

We let $\Lambda_{\grh}$ be the quotient of $X\times \hath$ were we
identify $(x,\rho)$ and $(y,\sigma)$ if
$\overline{\grh\cdot x} = \overline{\grh\cdot y}$ and $\sigma$ and
$\rho$ agree on $S_{x}$ (which is necessarily the same as $S_{y}$
since $\overline{\grh\cdot x} = \overline{\grh\cdot y}$ and $\grh$ is
abelian).  We let $[x,\rho]$ be the class of $(x,\rho)$ in
$\Lambda_{\grh}$.  As discussed following
\cite{wil:tams81}*{Corollary~5.9}, there is a continuous $\grg$-action
on $\Lambda_{\grh}$ given by
$g\cdot [x,\rho]=[g\cdot x, g\cdot \rho]$.

\begin{prop}[\cite{wil:tams81}*{\S5}] \label{prop-81-results} Let
  $\grg$, $\grh$, and $\Lambda_{\grh}$ be as above.
  \begin{enumerate}
  \item The natural map of $X\times\hath$ onto $\Lambda_{\grh}$ is
    open.
  \item The natural map
    $\alpha:X\times\hath\to (\grg\backslash \Lambda_{\grh})^{\sim}$ is
    an open map.
  \item $\alpha(x,\rho)\mapsto \psind(x,\rho\restr{G(x)})$ is a
    homemorphism of $(\grg\backslash \Lambda_{\grh})^{\sim}$ onto
    $\primcsg$.
  \item If $\alpha(x,\rho)=\alpha(y,\sigma)$, then there are
    $g_{n}\in\grg$ and $\rho_{n}\in \hath$ such that
    $(g_{n}\cdot x,\rho_{n})\to (y,\sigma)$ in $X\times\hath$ and such
    that $\rho_{n}$ and $g_{n}\cdot \rho$ agree on $S_{g_{n}\cdot x}$.
  \end{enumerate}
\end{prop}
\begin{proof}
  (a) This follows from \cite{wil:tams81}*{Theorem~5.3} and the
  discussion following its proof.
  
  (b) As in the first paragraph of the proof of
  \cite{wil:tams81}*{Corllary~5.11}, this follows from
  \cite{wil:tams81}*{Corollary~5.9} and part~(a).

  (c) This is \cite{wil:tams81}*{Corollary~5.10}.

  (d) If $\alpha(x,\rho)=\alpha(y,\sigma)$, then there are
  $g_{n}\in\grg$ such that $[g_{n}\cdot x,g_{n}\cdot \rho]$ converges
  to $[y,\sigma]$ in $\Lambda_{H}$.  Since the natural map of
  $X\times \hath$ onto $\Lambda_{H}$ is open by part~(a), we can pass
  to a subseqence, relabel, and assume that there are
  $(z_{n},\rho_{n})\in X\times \hath$ such that
  $(z_{n},\rho_{n})\to (y,\sigma)$ in $X\times\hath$,
  $\overline{H\cdot z_{n}}=\overline{H\cdot (g_{n}\cdot x)}$, and
  $g_{n} \cdot \rho$ and $ \rho_{n}$ agree on
  $S_{z_{n}}=S_{g_{n}\cdot x}$.  Since $X$ is a metric space, there
  are $h_{n}\in \grh$ such that $h_{n}g_{n}\cdot x\to y$.  Since
  $\grh$ is abelian and all the stability groups are contained in
  $\grh$, $h_{n}g_{n}\cdot \rho=g_{n}\cdot \rho$ and
  $S_{h_{n}g_{n}\cdot x}=S_{g_{n}\cdot x}$. Thus
  $(h_{n}g_{n}\cdot x,\rho_{n})\to (y,\sigma)$ in $X\times \hat \grh$
  and $h_{n}g_{n}\cdot \rho$ and $\rho_{n}$ agree on
  $S_{h_{n}g_{n}\cdot x}$.
\end{proof}

\begin{lemma}
  \label{lem-cts-stab} Suppose that
  $(x_{n},\rho_{n})\to (x_{0},\rho_{0})$ in $X\times\hath$.  Then
  $(x_{n},\rho_{n}\restr{S_{x_{n}}})\to (x,\rho_{0}\restr{S_{n}})$ in
  $\stabg$.
\end{lemma}
\begin{proof}
  For convenience, let $\pi_{n}=\rho_{n}\restr{S_{x_{n}}}$.  We
  clearly have $x_{n}\to x$.  If we have passed to a subsequence and
  relabeled, we can pass to another so that
  $S_{n_{k}}\to S \subset S_{x_{0}}$.  But then
  $(H(x_{n_{k}}), \pi_{n_{k}})\to (S,\pi_{0}\restr S)$ in $\Sigmah$.
  So the result follows from Proposition~\ref{prop-conv-ts}.
\end{proof}

\begin{prop}
  \label{prop-thesis-case} Suppose that $G=G(\grg,X)$ is an amenable
  second countable action groupoid such that the stability groups are
  all contained in a fixed abelian subgroup of $\grg$.  If $\psindq$
  is as in Proposition~\ref{prop-main-cts}, then
  $\psindq:\stabggt\to \primcsg $ is a homeomorphism.
\end{prop}

\begin{proof}
  Since $G$ is second countable and amenable, it is not hard to see
  that $C_{0}(X)\rtimes_{\lt}\grg$ is EH-regular (see
  \cite{ionwil:iumj09}).  Of course, since $\grh$ is abelian,
  $C_{0}(X)\rtimes_{\lt}\grh$ is necessarily EH-regular.  Furthermore,
  we may assume that $\grh$ is normal and apply the results from
  \cite{wil:tams81} via Proposition~\ref{prop-81-results}.

  To see that $\psind$ is open, we proceed as in the proof of
  Theorem~\ref{thm-jmp-dis}.  Suppose that $V$ is a neighborhood of
  $(x_{0},\pi_{0})$ in $\stabg$ and that
  $\psind(x_{n},\pi_{n})\to \psind(x_{0},\pi_{0})$.  Since we started
  with an arbitrary sequence, it will suffice to show that
  $\psind(x_{n},\pi_{n})$ has a subsequence which is eventually in
  $\psind(V)$.

  Since every character on $S_{x_{n}}$ is the restriction of some
  $\rho\in\hath$, we can replace $\pi_{n}$ with
  $\rho_{n}\restr{G(x_{n})}$ for some $\rho_{n}\in\hat\grh$.  Let
  $\alpha:X\times \hat \grh\to (\grg\backslash \Delta)^{\sim}$ be the
  natural map.  Then in view of Proposition~\ref{prop-81-results}(c),
  we can assume that
  $\alpha(x_{n},\rho_{n})\to \alpha(x_{0},\rho_{0})$.  Since $\alpha$
  is open, we can pass to a subsequence, and relabel, so that there is
  a sequence $(y_{n},\sigma_{n})\to (x_{0},\rho_{0})$ in
  $X\times\hat \grh$ with
  $\alpha(y_{n},\sigma_{n})=\alpha(x_{n},\rho_{n})$.  Since
  $X\times\grh$ is a metric space, we can use
  Proposition~\ref{prop-81-results}(d) to find $g_{n}\in \grg$ such
  that $ (g_{n}\cdot x_{n},\rho_{n}') \to (x_{0},\rho_{0})$ and such
  that for all $n$, $\rho_{n}'$ and $g_{n}\cdot \rho_{n}$ agree on
  $S_{g\cdot x_{n}}$.  It follows from Lemma~\ref{lem-cts-stab} that
  $g_{n}\cdot (x_{n},\rho_{n}\restr{S_{x_{n}}})\to
  (x,\rho_{0}\restr{S_{x}})=(x,\pi_{0})$ in $\stabg$.  Then
  $g_{n}\cdot (x_{n},\rho_{n}\restr{S_{x_{n}}})$ is eventually in $V$.
  Since $\psind$ is $\grg$-invariant, $\psind(x_{n},\rho_{n})$ is
  eventually in $\psind(V)$.  Thus, $\psind$ is open.

  To see that $\psindq$ is injective, suppose that
  $\psind(x,\rho\restr{S_{x}})=\psind(y,\sigma\restr{S_{y}})$.  Then
  $\alpha(x,\rho)=\alpha(y,\sigma)$ and we can use
  Proposition~\ref{prop-81-results}(d) to find $g_{n}$ such that
  $(g_{n}\cdot x,\rho_{n})\to (y,\sigma)$ in $X\times\grh$ such that
  $\rho_{n}$ and $g_{n}\cdot \rho$ agree on $S_{g_{n}\cdot x}$.  But
  then $g_{n}\cdot (x,\rho\restr{S_{n}})$ converges to
  $(y,\sigma\restr{S_{y}})$ in $\stabg$ by Lemma~\ref{lem-cts-stab}.
  Therefore
  $(y,\sigma\restr{S_{y}} \in \overline{G\cdot
    (x,\rho\restr{S_{x}})}$.  Since the situation is symmetric, we are
  done.
\end{proof}

\section{Generalized Proper Groupoids}
\label{sec:gener-prop-group}

Let $\pi:G\to \go\times\go$ be given by
$\pi(\gamma)=(r(\gamma),s(\gamma))$.  The image, $\rg=\pi(G)$, is an
equivalence relation on $\go$ and a topological groupoid when $\rg$ is
given the relative product topology.\footnote{It is more common to
  equip $\rg$ with the often finer quotient topology.  However, it is
  the relative product topology that is suitable for our needs here.}  Of
course, in general, $\rg$ may not be locally compact.  In fact, $\rg$
is closed in $\go\times\go$ if and only if $\gugo$ is Hausdorff
(\cite{wil:toolkit}*{Ex 2.1.10}).  Part of our interest in $\pi$ is
due to the following straightforward observation.

\begin{lemma}\label{lem-pi-acts}
  If $G$ has abelian isotropy and $H\in \so$, then $\gamma\cdot H$
  depends only on $\pi(\gamma)$.
\end{lemma}

It follows we get a well-defined action of $\rg$ on $\so$:
$\pi(\gamma)\cdot H=\gamma\cdot H$.

\begin{lemma}
  \label{lem-pi-cts} Suppose that the action of $\rg$ on $\so$ is
  continuous and that $H_{n}\to H_{0}$ in $\so$.  Let
  $u_{n}=p_{0}(H_{n})$.  If $(\gamma_{n})$ is a sequence in $G$ 
  such that $\gamma_{n}\cdot u_{n}\to \gamma_{0}\cdot u_{0}$, then
  $\gamma_{n}\cdot H_{n}\to \gamma_{0}\cdot H_{0}$
\end{lemma}
\begin{proof}
  The convoluted hypotheses just imply that
  $\pi(\gamma_{n})\to \pi(\gamma_{0})$ in $\rg$.
\end{proof}

Recall that a locally compact groupoid is a proper if it
acts properly on its unit space.  For example, the action groupoid
$G=G(\grg,X)$ is proper if and only if $\grg$ acts properly on $X$.
Note that if $G$ is a proper groupoid, then all of its isotropy groups
are compact.

\begin{lemma}\label{lem-proper-ok}
  Suppose that $G$ is a proper groupoid with abelian isotropy.  Then
  $\gugo$ is Hausdorff and $\rg$ acts continuously on~$\so$.
\end{lemma}
\begin{proof}
  If $G$ is proper, then $\gugo$ is Hausdorff
  \cite{wil:toolkit}*{Proposition~2.18}.  To see that the $\rg$-action
  is continuous, suppose that $H_{n}\to H_{0}$ in $\so$ with
  $u_{n}=p_{0}(H_{n})$.  Suppose
  $\pi(\gamma_{n})=(v_{n},u_{n})\to \pi(\gamma_{0})=(v_{0},u_{0})$.
  We need to verify that
  $\pi(\gamma_{n})\cdot H_{n}\to \pi(\gamma_{0})\cdot H_{0}$.  After
  passing to a random subsequence and relabeling, it will suffice to
  produce a subsequence with this property.  Since $u_{n}\to u_{0}$
  and $v_{n}=\gamma_{n}\cdot u_{n}\to v_{0}$, it follows from
  \cite{wil:toolkit}*{Proposition~2.17} that we can pass to a
  subsequence and assume $\gamma_{n}\to \eta_{0}$ with
  $\pi(\gamma_{0})=\pi(\eta_{0})$.  Since the action of $G$ on $\so$
  is continuous,
  $\gamma_{n}\cdot H_{n}\to \eta_{0}\cdot H_{0}=\gamma_{0}\cdot
  H_{0}$.  This completes the proof.
\end{proof}

Motivated in part by the previous lemma, we make the following definition.

\begin{definition}
  We say that $G$ is proper modulo its isotropy if $\gugo$ is
  Hausdorff and $\rg$ acts continuously on $\so$.
\end{definition}

To both justify our terminology, and to see that our next theorem has
wide applicability, we want to exhibit the following class of examples
based on the unit space fixing extensions studied extensively in
\cites{iksw:jot19,ikrsw:jfa20,ikrsw:xx21}. 

\begin{example}\label{ex-gen-proper}
  As usual, suppose that $G$ has abelian isotropy.  Suppose that $\G$
  is a proper groupoid and that there is a
  continuous open groupoid epimorphism $p:G\to\G$ that restricts to a
  homeomorphism of $\go$ with $\G^{(0)}$.  Then $\A=\ker
  p=\set{\gamma\in G:p(\gamma)\in\G^{(0)}}$ is contained in the
  isotropy subgroupoid $G'=\set{\gamma\in G:r(\gamma)=s(\gamma)}$.
  It follows that $\A$ is an abelian group bundle and 
  we have a diagram 
  \begin{equation}
    \label{eq:33}
    \begin{tikzcd}[column sep=3cm]
    \A \arrow[r,"\iota", hook] \arrow[dr,shift left, bend right = 15]
    \arrow[dr,shift right, bend right = 15]&G \arrow[r,"p", two
    heads] \arrow[d,shift left] \arrow[d,shift right]&\G
    \arrow[dl,shift left, bend left = 15] \arrow[dl,shift right, bend
    left = 15]
    \\
    &\go&
  \end{tikzcd}
\end{equation}
so that we can view $G$ as a unit space fixing extension of $\G$ by
$\A$.  Our next lemma confirms that such groupoids are always examples of
groupoids that are proper modulo their isotropy.
\end{example}

\begin{lemma}
  \label{lem-gen-proper}  Let $G$ be as in
  Example~\ref{ex-gen-proper}.  Then $G$ is proper modulo its isotropy.
\end{lemma}
\begin{proof}
  We can use $p\restr\go$ to identify
  $\go$ and $\G^{(0)}$.   Then
  $r(\gamma)=r(p(\gamma))$ and $s(\gamma)=s(p(\gamma))$.
  Then we can identify $\gugo$ and
  $\G\backslash \go$.  Since $\G$ is proper,  it follows $\gugo$ is
  Hausdorff as in Lemma~\ref{lem-proper-ok}.

  Note that $p(\gamma)=p(\eta)$ if and only if there is a $a\in \A$
  such that $\eta=\gamma a$.

  To see that $\rg$ acts continuously, we make the appropriate
  modifications to the second part of the proof of
  Lemma~\ref{lem-proper-ok}.  Suppose that $H_{n}\to H_{0}$ in $\so$
  with $u_{n}=p_{0}(H_{n})$.  Suppose
  $\pi(\gamma_{n})=(v_{n},u_{n})\to \pi(\gamma_{0})=(v_{0},u_{0})$.
  We need to see that after passing to a subsequence and relabeling,
  $\pi(\gamma_{n})\cdot H_{n}\to \pi(\gamma_{0})\cdot H_{0}$.

  Since $\G$ is proper, we can assume that
  $p(\gamma_{n})\to p(\eta_{0})$ and
  $\pi_{\G}(p(\gamma_{0}))=\pi_{\G}(p(\eta_{0}))$.  Since $p$ is open,
  we can pass to subsequence, relabel, and assume that there are
  $\eta_{n}\to \eta_{0}$ in $G$ such that $p(\eta_{n})=p(\gamma_{n})$.
  In particular, there are $a_{n}\in \A$ such that
  $\gamma_{n}a_{n}\to \eta_{0}$.  Thus
  $\gamma_{n}\cdot H_{n}=(\gamma_{n}a_{n})\cdot H_{n}\to
  \eta_{0}H_{0}=\gamma_{0}\cdot H_{0}$.
\end{proof}

\begin{remark}
    Notice that for all $u\in\go$, $\A(u)$ is a subgroup of $G(u)$.
    Since $\G(u)=p(G(u))$, we must have $G(u)/\A(u)$ compact.
    While we can always form the groupoid $G'\backslash G$, it may not
    even be locally compact if the isotropy map is not continuous.  
  \end{remark}

\begin{thm}
  \label{thm-baby} Suppose that $G$ is a second countable locally
  compact Hausdorff groupoid with a Haar system and abelian isotropy.
  Suppose that $G$ is proper modulo its isotropy.  Then $\irrindq$ is
  a homeomorphism of $\stabgg$ onto $\cs(G)^{\wedge}$.
\end{thm}

\begin{proof}[Proof of Theorem~\ref{thm-baby}]
  In view of Corollary~\ref{cor-main-cts}, it suffices to see that
  $\irrind$ is an open map.  Since $\cs(G)$ is GCR, we can identify
  $\cs(G)^{\wedge}$ and $\primcsg$, so it will suffice to see that
  $\psind$ (as defined in Proposition~\ref{prop-main-cts}) is open.
  As in the proof of Theorem~\ref{thm-jmp-dis}, it will suffice, given
  a sequence $\bigl(\psind(u_{n},\pi_{n})\bigr)$ converging to
  $\psind(u_{0},\pi_{0})$, to find a subsequence
  $\bigl((u_{n_{k}},\pi_{n_{k}})\bigr)$ and a sequence
  $\bigl((v_{k},\sigma_{k})\bigr)$ in $\stabg$ such that
  $(v_{k},\sigma_{k})\to (u_{0},\pi_{0})$ and
  $\psind(v_{k},\sigma_{k})=\psind(u_{n_{k}},\pi_{n_{k}})$.

  Using Corollary~\ref{cor-orbits}, we can pass to a subsequence,
  relabel, and assume that there are $\gamma_{n}\in G$ such that
  $\gamma_{n}\cdot u_{n}\to u_{0}$.  Since
  $\psind(\gamma_{n}\cdot u_{n},\gamma_{n}\cdot \pi_{n})
  =\psind(u_{n},\pi_{n})$, we can replace $(u_{n},\pi_{n})$ with
  $(\gamma_{n}\cdot u_{n},\gamma_{n}\cdot \pi_{n})$ and assume that
  $u_{n} \to u_{0}$.

  Suppose that there are infinitely many $n$ such that $[u_{n}]=[v]$
  for some $v\in \go$ (where $[u]$ denotes the orbit of $u$ in $\go$).
  Then we can pass to a subsequence, relabel, and assume that for all
  $n$ there are $\gamma_{n}\in G$ such that $\gamma_{n}\cdot v=u_{n}$.
  Then $\gamma_{n}\cdot v\to u_{0}$ and $u_{0}\in \overline{[v]}=[v]$.
  (Orbits are closed since $\gugo$ is Hausdorff.)  Hence for all $n$,
  $[u_{n}]=[u_{0}]=[v]$.  Replacing $(u_{n},\pi_{n})$ with
  $(\eta_{n}\cdot u_{n},\eta_{n}\cdot \pi_{n})$ for appropriate
  $\eta_{n}\in G$, we can assume $u_{n}=u_{0}$ for all $n$.

  Now we can replace $G$ with $G\restr{[u_{0}]}$ and assume
  $\irrind(u_{0},\pi_{n})\to \irrind(u_{0},\pi_{0})$ as irreducible
  representations of $\cs(G\restr{[u_{0}]})$.  Since
  $G\restr{[u_{0}]}$ is equivalent to $G(u_{0})$, this means
  $\pi_{n}\to \pi_{0}$ in $G(u_{0})^{\wedge}$ (by
  \cite{rw:morita}*{Theorem~3.29}).

  This provides the required lift in the case infinitely many orbits
  $[u_{n}]$ coincide.

  Otherwise, we can pass to a subsequence, relabel, and assume that
  $n>m\ge0$ implies $[u_{m}]\not=[u_{n}]$.  We can also assume
  $G(u_{n})\to H\subset G(u_{0})$.

  We let $C=\set{u_{n}:n\ge0}$ and $F=G\cdot C$.  We claim that $F$ is
  closed.  Suppose that $\gamma_{k}\cdot v_{k}\to v$ with each
  $v_{k}\in C$.  If $v_{k}=u_{j}$ for infinitely many $k$, then
  $v\in \overline{ [u_{j}]}=[u_{j}]$ and $v\in F$.  Otherwise, we can
  pass to a subsequence, relabel, and assume $v_{k}=u_{n_{k}}$ with
  $u_{n_{k}}\to u_{0}$.  Then $\bigl([u_{n_{k}}]\bigr)$ converges to
  both $[u_{0}]$ and $[v]$ in $\gugo$.  Since the latter is Hausdorff,
  this implies $[v]=[u_{0}]$ and $v\in F$ in this case as well.  This
  proves that $F$ is closed as claimed.

  In view of Lemma~\ref{lem-pi-acts}, we get a well-defined function
  $\sg:F\to\so$ given by
  \begin{equation}
    \label{eq:32}
    \sg(\gamma\cdot u_{n})=
    \begin{cases}
      \gamma\cdot G(u_{n})&\text{if $n\ge1$, and}\\
      \gamma\cdot H&\text{if $n=0$.}
    \end{cases}
  \end{equation}
  We claim that $\sg$ is continuous.  If not, there is a sequence
  $\gamma_{k}\cdot v_{k}\to \gamma_{0}\cdot v_{0}$ with each
  $v_{k}\in C$ and such that no subsequence of
  $\bigl(\gamma_{k}\cdot \sg(v_{k})\bigr)$ converges to
  $\gamma_{0}\cdot \sg(v_{0})$.  If there are infinitely many $k$ such
  that $v_{k}=u_{n}$, then we can pass to subsequence, relabel, and
  assume $\gamma_{k}\cdot u_{n}\to \gamma_{0}\cdot v_{0}$.  Then
  $\gamma_{0}\cdot v_{0}\in \overline{[u_{n}]}=[u_{n}]$.  Then we can
  replace $v_{0}$ with $u_{n}$ and adjust $\gamma_{0}$ so that
  $\gamma_{k}\cdot u_{n}\to \gamma_{0}\cdot u_{n}$.  Since $\gugo$ is
  Hausdorff, the Mackey-Glimm-Ramsay Dichotomy
  \cite{wil:toolkit}*{Theorem~2.27} implies
  $\gamma\mapsto \gamma\cdot u_{n}$ is an open map.  Hence we can pass
  to a subsequence, relabel, and assume that
  $\gamma_{k}\to \gamma_{0}$.  Then
  $\sg(\gamma_{k}\cdot u_{n})=\gamma_{k}\cdot \sg(u_{n})\to
  \gamma_{0}\cdot \sg(u_{n})=\sg(\gamma_{0}\cdot u_{n})$ which
  contradicts our choice of $\bigl(\gamma_{k}\cdot v_{k}\bigr)$.

  Now we assume that no $v_{k}$ is repeated infinitely often.  Then we
  must have $v_{k}\to u_{0}$.  Since
  $\gamma_{k}\cdot v_{k}\to \gamma_{0}\cdot v_{0}$ and $\gugo$ is
  Hausdorff, we can adjust $\gamma_{0}$ if need be and replace
  $\gamma_{0}\cdot v_{0}$ with $\gamma_{0}\cdot u_{0}$.  Since $\sg$
  is clearly continuous on $C$, we have $\sg(v_{k})\to \sg(u_{0})$.
  Since $\rg$ acts continuously by assumption,
  $\gamma_{k}\cdot \sg(v_{k})\to \gamma_{0}\cdot \sg(u_{0})$.  Again,
  this contradicts our assumptions on
  $\bigl(\gamma_{k}\cdot v_{k}\bigr)$.  Hence $\sg$ is continuous on
  $F$ as claimed.

  Therefore $\sg$ is an equivariant subgroup section for $G\restr
  F$. Since $F$ is closed and $G$-invariant, we can replace $G$ by
  $G\restr F$.  Let
  $P^{*}:\I(\cs(G\restr F))\to \I(C_{0}(\Sigmah_{F})$ be the
  corresponding map.  Since
  $\psind(u_{k},\pi_{k})\to \psind(u_{0},\pi_{0})$ in
  $\Prim\bigl(\cs(G\restr F)\bigr)$, we have
  $P^{*}(\psind(u_{k},\pi_{k}))\to P^{*}(\psind(u_{0},\pi_{0}))$ in
  $\I(C_{0}(\Sigmah_{F})$.  By Proposition~\ref{prop-pstar}, if
  $k\ge1$, then $P^{*}(\psind(u_{k},\pi_{k})$ is the ideal of
  functions in $C_{0}(\Sigmah_{F})$ vanishing on
  $G\restr F\cdot (G(u_{k}),\pi_{k})=G\cdot (G(u_{k}),\pi_{k})$.
  Similarly, $P^{*}(\psind(u_{0},\pi_{0}))$ is the ideal vanishing on
  $H\cdot (H,\pi_{0}\restr H)$.  We can use Lemma~\ref{lem-fix-8.38}
  to pass to a subsequence, relabel, so that there are
  $\gamma_{k}\in G$ such that
  $\gamma_{k}\cdot (G(u_{k}),\pi_{k})\to (H,\pi_{0}\restr H)$ in
  $\Sigmah$.  But then $(u_{k},\pi_{k})\to (u_{0},\pi_{0})$ in
  $\stabg$ (by Proposition~\ref{prop-conv-ts}).  Since
  $\psind(\gamma_{k}\cdot (u_{k} , \pi_{k}))=\psind(u_{k},\pi_{k})$,
  this completes the proof.
\end{proof}

Before closing out this section, we note that there are natural
circumstances where $\rg$ acts continuously without having to assume
either that $G$ is proper or invoking Example~\ref{ex-gen-proper}.

\begin{lemma}
  \label{lem-trans-grp} Let $(\grg,X)$ be a locally compact
  transformation group with $\grg$ abelian, and let $G=G(\grg,X)$ be
  the corresponding action groupoid as in
  Section~\ref{sec:abelian-group-action}.  Then $\rg$ acts
  continuously on $\so$.
\end{lemma}

\begin{remark}
  It would suffice for all the $S_{x}$ to be central in $G$.
\end{remark}

\begin{proof}
  As usual, we identify $\go=\set{(x,e,x)\in G:x\in X}$ with $X$. Then
  the isotropy group $G(x)$ is $\set{(x,h,x):h\in S_{x}}$ where
  $S_{x}=\set{h\in \grg:h\cdot x=x}$ is the stability group at $x$.
  Let $c:G\to \grg$ be the cocycle $c(y,g,x)=g$.  Note that $c$
  restricts to an isomorphism of $G(x)=\set{(x,g,x):x\in S_{x}}$ with
  $S_{x}$.  More generally, if $H\in\so$ and $H\subset G(x)$, then
  $c(H)$ is a subgroup of $\grg(x)$.  Using
  \cite{wil:toolkit}*{Lemma~3.22} for example, we have
  $H_{n}\subset G(x_{n})$ converging to $H\subset G(x)$ if and only if
  $x_{n}\to x$ and $c(H_{n})\to c(H)$ as subgroups of $\grg$.  Since
  \begin{equation}
    \label{eq:35}
    (y,g^{-1},x)(x,h,x)(x,g,y)=(y,g^{-1}hg,y)=(y,h,y)
  \end{equation}
  it follows that $\rg$ acts continuously on $\so$.
\end{proof}

Recall that a continuous surjection $f:X\to Y$ has \emph{local
  sections} if for each $y\in Y$ there is a neighborhood $V\subset Y$
and a continuous function $c:V\to X$ such that $f(c(z))=z$ for all
$z\in V$.

\begin{lemma}
  \label{lem-local-sections} Suppose that $\pi:G\to \rg$ has local
  sections.  Then $\rg$ acts continuously on $\so$.
\end{lemma}
\begin{proof}
  Suppose that $H_{n}\to H_{0}$ with $p_{0}(H_{n})=u_{n}$.  Suppose
  also that
  $\pi(\gamma_{n})=(v_{n},u_{n})\to \pi(\gamma_{0})=(u_{0},v_{0})$.
  Let $U$ be a neighborhood of $(v_{0},u_{0})$ and $c:U\to G$ a local
  section for $\pi$.  We can assume that
  $\bigl((v_{n},u_{n})\bigr)\subset U$ and let
  $\eta_{n}=c(v_{n},u_{n})$.  Then $\eta_{n}\to \eta_{0}$ and
  $\pi(\eta_{n})=\pi(\gamma_{n})$ for all $n\ge0$.  Then
  \begin{align}
    \pi(\gamma_{n})\cdot H_{n}
    &=\eta_{n}\cdot H_{n}\to \eta_{0}\cdot
      H_{0}=\pi(\gamma_{0})\cdot H_{0}.\qedhere
  \end{align}
\end{proof}

\section{Examples}
\label{sec:examples}

In order to see that Theorem~\ref{thm-jmp-dis} has new interesting
applications, we want to exhibit some elementary examples where the
isotropy is continuous except for jump discontinuities as in
Definition~\ref{def-jmp-dis}.  For ease of exposition, all of the
constructions here come from directed graphs, where the primitive
ideal space of the corresponding groupoid \cs-algebra can be computed
via \cite{honszy:jmsj04}.  Nevertheless, these examples show the
phenomena is reasonably common.

The groupoids we will construct will be \'etale groupoids and are now
usually called Deaconu--Renault groupoids.  The construction
originated from directed graphs in \cite{kprr:jfa97} and emerged on
its own in \cite{ren:otm00}.  For a more general construction, one can
consult \cite{renwil:tams16}*{\S5}.  We review the basics for the
constructions we need here.

We let $X$ be a locally compact space and let
$T:\dom(T)\subset X\to X$ be a local homeomorphism with $\dom(T)$ and
$\operatorname{ran}(T)=T(\dom(T))$ open subsets.  Let
$\Nz=\set{0,1,2,\dots}$.  Then
\begin{equation}
  \label{eq:48}
  G_{T}=\set{(x,m-n,y)\in X\times \Z\times
    X:\text{$T^{m}x=T^{n}y$ with $m,n\in\Nz$}}
\end{equation}
is a groupoid with respect to the natural operations.  The sets
\begin{equation}
  \label{eq:49}
  Z(U,m,n,V)=\set{(x,m-n,y)\in G_{T}:T^{m}x=T^{n}y}
\end{equation}
with $U$ and $V$ open in $X$ and $m,n\in\Nz$ form a basis for a
locally compact Hausdorff topology on $G_{T}$ with respect to which $G_{T}$ is
an \'etale groupoid, and such that the map $c:G_{\T}\to \Z$ given by
$c(x,k,y)=k$, is a continuous cocycle on $G_{T}$.  Furthermore, the
groupoid $G_{T}$ is always amenable (see
\cite{ren:otm00}*{Proposition~2.9} or
\cite{renwil:tams16}*{Theorem~5.13}) so that our results apply.
This construction
of $G_{T}$ from $X$ and $T$ is an example of a Deaconu--Renault
groupoid.

Given a general directed graph $E$, there is now a well understood way
of associating to it a Deaconu--Renault groupoid $G_{T}$ as above.  We let $X$
be the boundary path space $\partial E$ and $T$ the edge shift.  Although
we don't actually require them here, the
details of the construction
can be found in \cite{bcw:etds17} or \cite{web:pams14}.

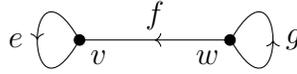
\begin{figure}[!ht]
  \centering
  \begin{tikzpicture}[decoration = {markings, mark = at position 0.5
      with {\arrow{>}} }]
    \filldraw (0,0) circle (2pt) node[below right]{$v$};
    \draw[scale=3,postaction=decorate] (0,0) to[in=240,out=120,loop]
    (0,0); \draw (-.85,0) node{$e$}; \draw (2.85,0) node{$g$};
    \draw[scale=3,postaction=decorate] (.67,0) to[in=60,out=300,loop]
    (.67,0); \filldraw (2,0) circle (2pt) node[below left]{$w$};
    \draw[postaction=decorate] (2,0) -- (0,0) node[above,pos=.5]{$f$};
  \end{tikzpicture}
  \caption{A Simple Graph} \label{fig:aidan}
\end{figure}

To start, let $E$ be the graph in Figure~\ref{fig:aidan} which comes
from \cite{simwil:art15}*{Example~3.4}. In this case, the boundary
paths are all infinite paths and consist of $e^{\infty}$,
$g^{\infty}$, and $g^{n-1}fe^{\infty}$ for $n\ge1$.  Topologically,
$\set{e^{\infty}} \cup \set{g^{n-1}fe^{\infty}}$ is discrete while
every neighborhood of $g^{\infty}$ contains $\set{g^{n-1}fe^{\infty}}$
for $n\ge N$ for some $N$.  In particular,
$\lim_{n\to\infty} g^{n-1}fe^{\infty}=g^{\infty}$.  At this point, we
can forget about the graph and
reduce to simply specifying a space $X$ and local homeomorphism $T$.
Then we can construct $G_{T}$ as above.

\begin{example}
  [\cite{simwil:art15}*{Example~3.4}] \label{ex-aidan} We let
  $X=\partial E$ and think of $x_{n}= g^{n-1}fe^{\infty}$ as a point
  in $X$ for $n\ge1$.  Then $X$ is
  the space $\set{x_{0},x_{1},x_{2},\dots,x_{\infty}}$ where
  $X\setminus \sset{x_{\infty}}$ is discrete and the open
  neighborhoods of $x_{\infty}$ are all of the form
  $\set{x_{k}:k\ge N}$ for some integer $N\ge0$.  The map $T:X\to X$ is
  given by
  \begin{equation}
    \label{eq:34}
    T(x_{k})=
    \begin{cases}
      x_{0}&\text{if $k=0$,} \\ x_{k-1}&\text{if $1\le k<\infty$, and}
      \\
      x_{\infty}&\text{if $k=\infty$.}
    \end{cases}
  \end{equation}

  Since $T^{k}x_{k}=x_{0}$ for all $0\le k <\infty$, we have
  $(x_{0},-k,x_{k})\in G_{T}$ and $[x_{0}]=\set{x_{k}:0\le k<\infty}$.
  On the other hand, $[x_{\infty}]=\sset{x_{\infty}}$.  Note that
  $[x_{0}]$ is dense and the orbit space $G_{T}\backslash X$ is
  $T_{0}$ consisting of the dense point $[x_{0}]$ and the closed point
  $[x_{\infty}]$.

  If $0\le k <\infty$, then $T^{n}x_{k}=x_{0}$ for all $n\ge k$.
  Hence $(x_{k},m+k-(n+k),x_{k})\in G_{T}$ for all $m,n\in\N$.  Thus
  $G_{T}(x_{k})=\set{(x_{k},l,x_{k}):l\in\Z}$ for all $0\le k<\infty$.
  This clearly holds for $x_{\infty}$ as well.

  Now let $(x_{\infty},m-n,x_{\infty})\in G_{T}(x_{\infty})$ and let
  $Z(U,m,n,V)$ be a basic open neighborhood of
  $(x_{\infty},m-n,x_{\infty})$.  Let $k> \max\set{m,n}$.  Then
  $T^{m}x_{k}=T^{n}x_{k}$ only if $m=n$.  Thus if $m\not=n$, then
  $G_{T}(x_{k})\cap Z(U,m,n,V)=\emptyset$ if $k>\max\set{m,n}$.  It
  follows that $G_{T}(x_{k})\to \sset{(x_{\infty},0,x_{\infty})}$.

  Since $\set{x_{k}:0\le k<\infty}$ is discrete, we have established
  that $G_{T}$ has isotropy which is continuous except for a jump
  discontinuity at $x_{\infty}$.

  Now we can identify $\stab(G_{T})$ with $X\times \T$ and the sets
  $\sset {x_{k}}\times\T$ are each homeomorphic to $\T$ for all $k$
  and open if $k<\infty$.  Consider a sequence
  $\bigl((x_{k},z_{k})\bigr)$.  Recall that $\bigl(G_{T}(x_{k})\bigr)$
  converges to the trivial subgroup of $G_{T}(x_{\infty})$.  Thus if
  $a_{k}\in G_T(x_{k})$ and $a_{k}$ converges, it must converge to the
  identity in $G_T(x_{\infty})$.  Since the cocycle $c:G_{T}\to\Z$ is
  continuous, $a_{k}$ must eventually be the identity of
  $G_T(x_{k})$. It follows that $\bigl((x_{k},z_{k})\bigr)$ converges
  to $(x_{\infty},w)$ for all $w\in\T$ by
  Proposition~\ref{prop-conv-ts}.  \qed
\end{example}

The construction of the boundary path space is a bit more complicated
if we allow our graph to have sinks.  For example, let $E$ be the
graph in Figure~\ref{fig:danie}.  The boundary path space for the
graph is
\begin{equation}
  \label{eq:36}
  \partial E = \set{(e_1e_2)^n h, e_2(e_1e_2)^n h, (e_1e_2)^\infty,
    (e_2e_1)^\infty :n\in \Nz}.
\end{equation}
As in the previous example, the construction of $G_{T}$ becomes a bit
more transparent if we just describe the boundary path space as a
space without reference to the graph.

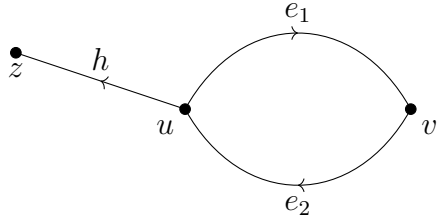
\begin{figure}[ht!]
  \centering
  \begin{tikzpicture}[decoration = {markings, mark = at position 0.5
      with {\arrow{>}} }]
    \draw[scale=3,postaction=decorate] (0,0) .. controls (.25,.45) and
    (.75,.45) ..  (1,0) node[above, pos =.5]{$e_{1}$};
    \draw[scale=3,postaction=decorate] (1,0) .. controls (.75,-.45)
    and (.25,-.45) ..  (0,0) node[below, pos =.5]{$e_{2}$};
    \draw[scale=3,postaction=decorate] (0,0) -- (-.75,.25)
    node[above,pos=.5]{$h$}; \filldraw (0,0) circle (2pt) node[below
    left]{$u$}; \filldraw (3,0) circle (2pt) node[below right]{$v$};
    \filldraw (-2.25,.75) circle (2pt) node[below]{$z$};
  \end{tikzpicture}
  \caption{A Graph With a Sink}
  \label{fig:danie}
\end{figure}

\begin{example}
  \label{ex-sink} We let $x_{n}=(e_{1}e_{2})^{n}h$ for $n\in \Nz$,
  $x_{\infty}= (e_{1}e_{2})^{\infty}$, $y_{ n}=e_{2}(e_{1}e_{2})^{n}h$
  for $n\in \Nz$, and $y_{\infty}=(e_{2}e_{1})^{\infty}$.  Then
  \begin{equation}
    \label{eq:39}
    X=\set{x_{0},x_{1},\dots,x_{\infty}} \cup
    \set{y_{0},y_{1},\dots,y_{\infty}} 
  \end{equation}
  where $X\setminus\set{x_{\infty},y_{\infty}}$ is discrete and
  $\set{x_{k}:k\ge N}$ and $\set{y_{k}:k\ge N}$ are neighborhoods of
  $x_{\infty}$ and $y_{\infty}$, respectively, for all $N\in \N$.  If
  we let $\dom(T)=X\setminus \sset {x_{0}$}, then we define a local
  homeomorphism $T:\dom(T)\subset X\to X$ by
  \begin{equation}
    \label{eq:40}
    T(x)=
    \begin{cases}
      y_{k-1}&\text{if $x=x_{k}$ with $1\le k <\infty$},
      \\
      y_{\infty}&\text{if $x=x_{\infty}$}, \\
      x_{k}&\text{if $x=y_{k}$, with $0\le k <\infty$, and}\\
      x_{\infty}&\text{if $x=y_{\infty}$.}
    \end{cases}
  \end{equation}
  Thus there are only two orbits in $\gto$; namely, the dense orbit
  $[x_{0}]=\set{x_{k},y_{k}:0\le k<\infty}$ and the closed orbit
  $[x_{\infty}]= \set{x_{\infty},y_{\infty}}$.  Clearly, the orbit
  space is $T_{0}$.

  If $k<\infty$, then $T^{n}x_{k}$ is only defined if $n\le 2k$.  Then
  it is not hard to see that
  \begin{equation}
    \label{eq:43}
    G_T(x_{k})=\set{(x_{k},0,x_{k})} \quad\text{for all $0\le k<\infty$.}
  \end{equation}
  Similarly, $G_T(y_{k})$ is trivial for all $0\le k<\infty$.  On the
  other hand,
  \begin{equation}
    \label{eq:44}
    G_T(x_{\infty})=\set{(x_{\infty},2k,x_{\infty}):k\in \Nz}
  \end{equation}
  and a similar formula holds for $G_T(y_{\infty})$.  Hence is
  isotropy is trivial except at $D=\set{x_{\infty},y_{\infty}}$ and
  $G_{T}$ has continuous isotropy except for jump discontinuities as
  in Remark~\ref{rem-free-off-d}.

  Furthermore, the sequence $\bigl((x_{k},1)\bigr)$ converges to
  $(x_{\infty},w)$ in $\stabgt$ for all $w\in\T$.  Similarly,
  $\bigl((y_{k},1)\bigr)$ converges to $(y_{\infty},w)$ for all
  $w\in\T$. \qed
\end{example}

We can now combine the phenomena exhibited by the previous two
examples by starting with a loop and then attaching exits to sinks or
additional loops. Consider the graph in Figure~\ref{fig:danie2}.  Now
\begin{equation}
  \label{eq:45}
  \partial E=\set{g^{n-1}f(e_{2}e_{1})^{\infty},
    g^{\infty},(e_{2}e_{1})^{\infty}, e_{2}(e_{1}e_{2})^{n-1}h,
    (e_{1}e_{2})^{\infty}, (e_{1}e_{2})^{n}h}.
\end{equation}
Now let $x_{\infty}=g^{\infty}$,
$x_{n}=g^{n-1}f(e_{2}e_{1})^{\infty}$, $x_{0}=(e_{2}e_{1})^{\infty}$,
$z_{n}=e_{2}(e_{1}e_{2})^{n-1}h$, $y_{\infty}=(e_{1}e_{2})^{\infty}$,
and $y_{n}=(e_{1}e_{2})^{n}h$.

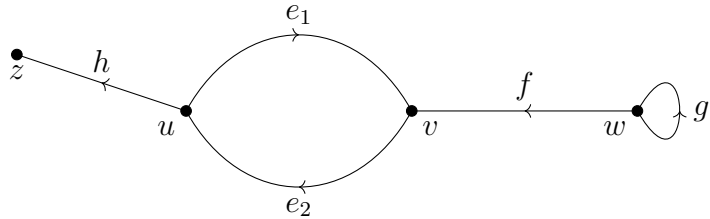
\begin{figure}[ht!]
  \centering
  \begin{tikzpicture}[decoration = {markings, mark = at position 0.5
      with {\arrow{>}} }]
    \draw[scale=3,postaction=decorate] (0,0) .. controls (.25,.45) and
    (.75,.45) ..  (1,0) node[above, pos =.5]{$e_{1}$};
    \draw[scale=3,postaction=decorate] (1,0) .. controls (.75,-.45)
    and (.25,-.45) ..  (0,0) node[below, pos =.5]{$e_{2}$};
    \draw[scale=3,postaction=decorate] (0,0) -- (-.75,.25)
    node[above,pos=.5]{$h$}; \filldraw (0,0) circle (2pt) node[below
    left]{$u$}; \filldraw (3,0) circle (2pt) node[below right]{$v$};
    \filldraw (-2.25,.75) circle (2pt) node[below]{$z$};

    \draw[scale=3,postaction=decorate] (2,0) -- (1,0)
    node[above,pos=.5]{$f$}; \draw[scale=3,postaction=decorate] (2,0)
    to[in=60,out=300,loop] (2,0); \filldraw (6,0) circle (2pt)
    node[below left]{$w$}; \draw (6.85,0) node{$g$};
  \end{tikzpicture}
  \caption{A Combination}
  \label{fig:danie2}
\end{figure}

\begin{example}
  \label{ex-combo} We let
  \begin{equation}
    \label{eq:46}
    X=\set{x_{0},x_{1},\dots,x_{\infty}} \cup \set{z_{0},z_{1},\dots}
    \cup \set{y_{0},y_{1},\dots,y_{\infty}}
  \end{equation}
  where $X\setminus \set{x_{\infty},x_{0},y_{\infty}}$ is discrete and
  $\set{x_{k}:k\ge N}$, $\set{z_{k},k\ge N}$ and $\set{y_{k}:k\ge N}$
  are neighborhoods of $x_{\infty}$, $x_{0}$, and $y_{\infty}$,
  respectively, for all $N\in\N$.  Here
  $\dom(T)=X\setminus \sset{y_{0}}$ and
  \begin{equation}
    \label{eq:47}
    T(x)=
    \begin{cases}
      x_{\infty}& \text{if $x=x_{\infty}$,} \\
      x_{k-1} &\text{if $x=x_{k}$ for $1\le k <\infty$,} \\
      y_{\infty}&\text{if $x=x_{0}$,} \\
      y_{k-1}&\text{if $x=z_{k}$ for $1\le k<\infty$,} \\
      y_{0}&\text{if $x=z_{0}$,} \\
      x_{0}&\text{if $x=y_{\infty}$, and} \\
      z_{k}&\text{if $x=y_{k}$ for $1\le k <\infty$.}
    \end{cases}
  \end{equation}
  The orbits consist of $[x_{\infty}]=\sset{x_{\infty}}$,
  $[x_{0}]=\set{x_{k}:0\le k <\infty}\cup\sset{y_{\infty}}$, and
  $[z_{0}]=\set{z_{k}, y_{k}:0\le k<\infty}$. 
  As in the
  previous examples, it is not hard to see that $G_T(z_{k})$ and
  $G_T(y_{k})$ are trivial if $1\le k<\infty$ while $G_T(y_{\infty})$
  and $G_T(x_{k})$, for $0\le k \le \infty$ are all isomorphic to
  $\Z$.  However $G_T(x_{k})\to \set{(x_{\infty},0,x_{\infty})}$.

  To see that $G_{T}$ has continuous isotropy except for jump
  discontinuities, we define $\sg:\gto\to\so$ by
  \begin{equation}
    \label{eq:50}
    \sg(x)=
    \begin{cases}
      \sset{(x_{\infty},0,x_{\infty})} &\text{if $x=x_{\infty}$,}\\
      \sset{(y_{\infty},0,y_{\infty})}&\text{if $x=y_{\infty}$, and} \\
      G_T(x)&\text{otherwise.}
    \end{cases}
  \end{equation}

  It follows that any sequence of the form $\bigl((x_{k},z_{k}) \bigr)$
  converges to $(x_{\infty},w)$ in $\stabgt$ for any $w\in\T$.  On the
  other hand, $\bigl((z_{k},1)\bigr)$ converges to $(x_{0},w)$ in
  $\stabgt$ for all $w\in \T$.  Similarly, $\bigl((y_{k},1)\bigr)$
  converges to $(y_{\infty},w)$ in $\stabgt$ for all $w\in\T$. \qed
\end{example}

These examples make it clear that more complicated examples of
groupoids with continuous isotropy except for jump discontinuities can
be built by starting with a loop and then attaching sinks or
additional loops.


\def\noopsort#1{}\def\cprime{$'$} \def\sp{^}

\end{document}